\documentclass[final,1p,times,number]{elsarticle}
\usepackage[colorlinks]{hyperref}

\usepackage{amssymb}
\usepackage[title]{appendix}
\usepackage{tikz-cd} 
\usepackage{amsfonts,amsmath,amscd,amsthm,color,comment}
\usepackage[all]{xy}
\usepackage[normalem]{ulem} 
\usepackage{verbatim}
\newfont{\eufm}{eufm10 scaled\magstep1}

\newcommand{\cC}{\mathcal{C}}
\newcommand{\cD}{\mathcal{D}}
\newcommand{\cE}{\mathcal{E}}

\newcommand{\cR}{\mathcal{R}}
\newcommand{\cS}{\mathcal{S}}

\newcommand{\bbN}{\mathbb{N}}

\newcommand{\bbC}{\mathbb{C}}
\newcommand{\bbD}{\mathbb{D}}
\newcommand{\bbK}{\mathbb{K}}

\newcommand{\bbR}{\mathbb{R}}

\def\dres{\partial{\rm Res}}

\def\ord{\rm ord}

\def\KdV{\rm KdV}
\def\kdv{\rm kdv}

\def\Sing{\rm Sing}

\def\para{\vspace{1.5 mm}}

\newtheorem{thm}{Theorem}[section]
\newtheorem{lem}[thm]{Lemma}

\newtheorem{prop}[thm]{Proposition}
\newtheorem{defi}[thm]{Definition}
\newtheorem{rem}[thm]{Remark}

\newtheorem{ex}[thm]{Example}

\begin{document}

\begin{frontmatter}

\title{Spectral Picard-Vessiot fields \\ for Algebro-geometric Schr\" odinger operators}

\author{Juan J. Morales-Ruiz}
\address{Dpto. de Matem\' atica Aplicada. E.T.S. Edificaci\' on. Avda. Juan de Herrera 6.\\
Universidad Polit\' ecnica de Madrid. 28040, Madrid. Spain}
\ead{juan.morales-ruiz@upm.es}

\author{Sonia L. Rueda}
\address{Dpto. de Matem\' atica Aplicada. E.T.S. Arquitectura. Avda. Juan de Herrera 4.\\
Universidad Polit\' ecnica de Madrid. 28040, Madrid. Spain}
\ead{sonialuisa.rueda@upm.es}

\author{Maria-Angeles Zurro}
\address{Dpto. de Matem\' aticas. Facultad de Ciencias. Ciudad Universitaria de Cantoblanco.\\
Universidad Aut\'onoma de Madrid. E-28049 Madrid. Spain}
\ead{mangeles.zurro@uam.es}



\begin{abstract}

This work is a galoisian study of the spectral problem $L\Psi=\lambda\Psi$, for algebro-geometric second order differential operators $L$, with coefficients in a differential field, whose field of constants $C$ is algebraically closed and of characteristic zero. Our approach  regards the spectral parameter $\lambda$ an algebraic variable over $C$, forcing the consideration of a new field of coefficients for $L-\lambda$,  whose field of constants is the field $C(\Gamma)$ of the spectral curve $\Gamma$. Since $C(\Gamma)$   is no longer algebraically closed,  the need arises of a new algebraic structure, generated by the solutions of the spectral problem over $\Gamma$,  called ``Spectral Picard-Vessiot field" of $L-\lambda$. An existence theorem is proved using differential algebra, allowing to recover classical Picard-Vessiot theory for each $ \lambda = \lambda_0 $.  For rational spectral curves, the appropriate algebraic setting is established to solve $L\Psi=\lambda\Psi$  analitically and to use symbolic integration. We illustrate our results for Rosen-Morse  solitons.

\end{abstract}


\end{frontmatter}




\section{ Introduction}

\setcounter{equation}{0}

Algebro-geometric operators are deeply linked to the integrabillity of partial differential equations of solitonic type, a review on this subject can be found in \cite{BB}, \cite{GH}, \cite{GGKM}. In this paper we describe the differential field structure generated by the solutions of the spectral problem 
\begin{equation}\label{eq-problem}
    L\Psi=\lambda\Psi,  
\end{equation}
for an algebro-geometric second order operator $L$, with coefficients in a differential field $(\Sigma,\partial)$, whose field of constants $C$ is algebraically closed and of characteristic zero. The novelty of our approach is to regard the spectral parameter  $\lambda$ to be an algebraic variable over $C$. This forces the consideration of a carefully chosen new field of coefficients for $L-\lambda$, whose field of constants provides the  natural constants of this spectral problem, the field of rational functions on a plane algebraic curve, the so called spectral curve. 

\para

{The goal of this work is a galoisian study of the algebro-geometric spectral problem \eqref{eq-problem}}.
One of our guiding ideas is the strong connection between the integrability (ie, solvability in closed form) of the direct and the inverse problems for the Schr\"{o}dinger equation. By ``direct problem" we understand, given the potential to obtain the eigenfunctions and the eigenvalues. The ``inverse problem " would be to obtain the potential from some suitable spectral data. In
fact, this was also the motivation for Drach in his 1919 papers \cite{Drach2, Drach3}, about the integrability in
closed form of the equation
\[
\dfrac{d^2 y}{dx^2}= (\varphi(x) + h)y \quad (I)
\]
where $h$ is the spectral parameter. So he wrote in \cite{Drach2}:

\para 

{\it ``Il est donc tr\`es important de conna\^itre les cas o\`u une simplification se presente dans l'int\'egration
de (I), en laissant le param\`etre h arbitraire. Nous avons r\'eussi \`a d\'eterminer la fonction $\varphi$ dans
tous les cas o\`u l'int\'egrale y peut s'obtenir par quadratures... ".
}
\para

\noindent Moreover he said that to study this problem
it is possible to use the classical theory of Picard about linear equations (Picard-Vessiot theory).
In other words, he considered the integrability of the direct problem in the sense of the Picard-Vessiot theory.  One could use the classical Picard-Vessiot theory \cite{VPS} for each choice of $\lambda=\lambda_0\in C$ and obtain the minimal field extension of the coefficient field of $L-\lambda_0$ that contains the solutions of $L\Psi=\lambda_0\Psi$.
A first attempt to use classical Picard-Vessiot theory was made by Y. Brezhnev in \cite{Brez3} that, following the ideas of Drach \cite{Drach2, Drach3}, exhibited formulas for the solutions $\Psi=\Psi(x,\lambda)$ by means of theta functions.
But along his papers Drach indicates that the constants are functions of the spectral parameter h.

\para

The alternative we present considers $\lambda$ as an algebraic parameter over the coefficient field $K$, which is not a free parameter but verifies the equation of an algebraic plain curve, the spectral curve, defined by a polynomial $f(\lambda,\mu)=0$. Our construction allows new achievements in the study of algebro-geometric operators, as we explain next.

\para 

{Algebro-geometric operators \cite{We} are characterized in this paper by having a nontrivial centralizer}. Moreover, in the case of second order operators we can prove that the centralizer is the ring $C[L,A]$, for an appropriate minimal odd order operator $A$, which is the ring $C(\Gamma )$ of an affine plane algebraic curve $\Gamma$. The curve $\Gamma$ is the celebrated {\it spectral curve} discovered by Burchnall and Chaundy in their visionary article \cite{BC}, where a correspondence was established between algebraic curves and pairs of commuting ordinary differential operators. 
Thus for algebro-geometric operators the spectral parameter $\lambda$ is not a free parameter since  $f(\lambda,\mu)=0$, and it also provides an algebraic relation between the operators $ L $ and $ A $,  $f(L,A)=0$. 

\para 

In Section \ref{sec-alg-geom} we explain  the connection between algebro-geometric operators, non trivial centralizers and the Korteveg de Vries (KdV) hierarchy of differential equations. More precisely, algebro-geometric second order opera\-tors in normal form {(i.e. with $0$ coefficient in $\partial$)} are Schr\" odinger operators $L=-\partial^2+u$ with potential $u$ verifying one of the equations of the KdV hierarchy (see Appendix \ref{sec-KdV}). In fact we establish the following result.

\para 

{
{\bf Theorem A. }
{\it Given $L_s = -\partial^2+u_s$ the following statements are equivalent.}
\begin{enumerate}
    \item {\it $L_s$ is algebro-geometric.}
    \item {\it There exists a unique monic operator $A_{2s+1}$ of minimal order $2s+1$  such that $\cC(L_s)=C[L_s,A_{2s+1}]$ and $A^2_{2s+1}+R_{2s+1}(L_s)=0$, with $R_{2s+1}(\lambda)$ in $C[\lambda]$ of degree $2s+1$.}
    \item {\it $u_s$ is a $KdV$-potential of KdV level $s$ (i.e. it satisfies one of the $\KdV_s$ equations of the KdV-hierarchy, see Appendix \ref{sec-KdV}).}
\end{enumerate}}

Hence, without loss of generality, we restrict to this case. The spectral curve is then defined by $f_s (\lambda ,\mu )=\mu^2+R_{2s+1}(\lambda)=0$ and it is appropriate to  consider  {the smallest differential field $K=C\langle u_s \rangle$ containing $ u_s $ and $ C $,} as the coefficient field of $L_s$, to study the {galoisian properties of $L_s-\lambda$}. {Furthermore, the} operator $L_s -\lambda$ determines as coefficient field an extended field of the curve, more precisely the fraction field $K(\Gamma_s )$ of the domain $K[\Gamma_s ]=K[\lambda,\mu]/(f_s)$.
{An important contribution of this paper is to establish its field of constants in Theorem \ref{thm-constants}.} 

\para

{\bf Theorem B.} 
{\it The field of constants of $K(\Gamma_s )$ is $C(\Gamma_s )$. }

\para

In Section \ref{sec-Picard-Vessiot}, we construct the minimal field extension $\cE$ of $K(\Gamma_s)$ containing the solutions of \eqref{eq-problem}. We call this new differential field structure {\it spectral Picard-Vessiot} field over the curve $\Gamma_s$,  since
\[
C(\Gamma_s )\subset K (\Gamma_s  )\subset \cE \ .
\]
{The field of constants $C(\Gamma_s)$ attached to the spectral coupled problem
\begin{equation}\label{eq-problem2}
     L_s\Psi=\lambda\Psi \quad , \quad  A_{2s+1}\Psi=\mu\Psi \ ,
\end{equation}
is no longer an algebraically closed field, forcing an adapted new Picard-Vessiot theory where the structure of the spectral curve plays an essential role. 
}

\para

The differential algebra theory  developed for the algebraic integration of the Risch differential equation in \cite{Bron}, by M. Bronstein and other authors,  was essential to prove the main results {of this paper}. {The Subresultant Theorem \ref{thm-subres},  in Appendix \ref{subsec-subres}, guarantees the existence of an ``intrinsic" right (common) factor $\partial-\phi_s$ of $L_s -\lambda$ (and $A_s-\mu$) in $K(\Gamma_s )[\partial]$, linked to the uniqueness of $A_{2s+1}$}. 
{
The common solution of \eqref{eq-problem2}, is the transcendental element $\Psi_s$, which is the  hyperexponential defined by $\partial \Psi_s=\phi_s\Psi_s$ over $K(\Gamma_s )$.}

\para

{The main result of this paper is the following existence theorem for spectral Picard-Vessiot extensions. It is based on Theorem \ref{thm-Bron}, where we prove that $\cE$ is a transcendental Liouvillian extension $K(\Gamma_s)\langle \Psi_s\rangle$ of $K(\Gamma_s)$, whose field of constants is the field of the curve $C(\Gamma_s)$. 
}

\para

{\bf Theorem C.} 
{\it 
Let $L_s$ be an algebro-geometric Schr\" odinger operator with spectral curve $\Gamma_s$. Let us consider  $\partial-\phi_s$, the intrinsic right factor of $L_s-\lambda$ in $K(\Gamma_s)[\partial]$, and a nonzero solution $\Psi_s$
defined by the differential relation $\partial(\Psi_s)=\phi_s\Psi_s$.
Then $K(\Gamma_s)\langle \Psi_s\rangle$ is a spectral Picard-Vessiot field over the curve $\Gamma_s$ of the equation $(L_s-\lambda)(\Psi)=0$.
}
\para

{Traditionally the spectral curve was assumed to be non-singular \cite{K77}, \cite{GH}, \cite{Brez3}. Nevertheless Burchnall and Chaundy in their 1931 paper \cite{BC2} realized that spectral curves with singularities need a special treatment. They studied the case of cuspidal curves, which are in fact singular rational curves, defined by $\mu^2-\lambda^{2s+1}=0$ for a particular type of Schr\" odinger operators. }

\para 

{We show how the study of \eqref{eq-problem} by means of classical Picard-Vessiot theory, for $\lambda=\lambda_0\in C$, is recovered. It is important to note that the spectral curve may be singular, to consider the spectral problem \eqref{eq-problem2} at each point $P_0=(\lambda_0,\mu_0)$ of $\Gamma_s$, that is}
\begin{equation}\label{eq-intro-problemesp}
     L_s\Psi=\lambda_0\Psi,\,\,\,  A_{2s+1}\Psi=\mu_0\Psi.
\end{equation}
We give the classical Picard-Vessiot extension $\cE_{P_0}$ of $K$, the coefficient field of $L-\lambda_0$, in Section \ref{sec-specialization}, Theorems \ref{thm-classPV1} and \ref{thm-classPV2}, distinguishing the treatment of singular and non-singular points of $\Gamma$. The differential algebra developed to solve the Risch differential equation in \cite{Bron} is again the key to prove these results. For all nonsingular points, but a finite number, this is the Liouvillian extension $K\langle y_0\rangle$ by  a transcendental element $y_0$, which is the common solution of \eqref{eq-intro-problemesp}. In fact $\partial y_0=\phi_0 y_0$,  where the common factor  $\partial-\phi_0=\partial-\phi_s(P_0)$ of $L-\lambda_0$ and $A-\mu_0$ is obtained by the specialization to $P_0$ of the common factor of problem \eqref{eq-problem2}. Moreover, for $u=u(x)\in \cC^{\infty}(\bbR)$ it is the well known Baker-Akheizer function $y_0 = \Psi(P_0,x,x_0)$ in \cite{GH}. Furthermore, at singular points we obtain a description of  the sequence of differential field extensions of $K$  to obtain $\cE_{P_0}$ (see Theorem  \ref{thm-classPV2}).

\para

In Section \ref{sec-Parametric} we restrict to the case of rational spectral curves, where the field of the curve is the field of rational functions $C(\tau)$ in an algebraic parameter $\tau$.
Considering a rational parametrization of $\Gamma$, say $\aleph (\tau)=(\chi_1(\tau),\chi_2(\tau))  $
, the spectral problem \eqref{eq-problem} in one-parameter form is
\begin{equation}\label{eq-problemparametrico}
    L\Psi=\chi_1(\tau)\Psi \ .
\end{equation}
The chosen parametrization establishes an isomorphism between $K(\Gamma)$ and $K(\tau)$, which is now the coefficient field of $L-\chi_1(\tau)$. 

\para 

More precisely, 
the isomorphism establised by the parametrization provides a right factor $\partial-\tilde{\phi}_s$ of $L_s-\chi_1(\tau)$.
In Section \ref{sec-Parametric}, we show that the spectral Picard-Vessiot field of $L_s-\lambda$ is isomorphic to a Liouvillian extension $K(\tau)\langle \Upsilon_s\rangle$ of the coefficient field $K(\tau)$ of $L-\chi_1(\tau)$, by a transcendental element $\Upsilon_s$, see Theorem \ref{thm-Bron2}. We prove the next result.

\para 

{\bf Theorem D.} 
{\it 
Let $L_s$ be an algebro-geometric Schr\" odinger operator with rational spectral curve $\Gamma_s$.
The Liouvillian extension $K(\tau)\langle \Upsilon_s\rangle$ of $K(\tau)$, by a nonzero solution $\Upsilon_s$ of  $(\partial-\tilde{\phi}_s)\Upsilon=0$
is isomorphic to a spectral Picard-Vessiot field over the curve $\Gamma_s$ of the equation $(L_s-\lambda)(\Psi)=0$.
}

\para

We can finally show the advantages of constructing solutions of the spectral problem \eqref{eq-problem} using a global rational parametrization of the spectral curve in $\bbC(\tau)^2$, instead of a local parametrization by Puiseux series in  $\bbC\langle\langle \tau \rangle\rangle^2$.  The coefficient field $K(\tau)$, where $\tau$ is now a free parameter, allows to say much more about the hyperexponential $\Upsilon_s$. Whenever $u_s$ is transcendental over $C(\tau)$ the algebraic integration algorithms in \cite{Bron} would allow to compute $\Upsilon_s$, see Remark \ref{rem-symbolicint}.
We illustrate this fact by means of a family of Rosen-Morse potentials in Example \ref{example}.

\para

 Moreover, we established the appropriate algebraic setting to solve the spectral problem \eqref{eq-problem} {analytically} for rational curves, possibly with singularities. Whenever the potential $u=u(x)$ is an analytic potential in some complex domain, we describe the analytic character of the common solution of problem \eqref{eq-problemparametrico} in Theorem \ref{thm-analytic}.

\para 

{To finish, recall that for an algebro-geometric potential $u_s$, the spectral parameter $\lambda$ of pro\-blem (1) is not a free parameter.
On the contrary,} the original spectral problem \eqref{eq-problem} in one-parameter form \eqref{eq-problemparametrico}, where $\tau$ is a free parameter, falls into the recently developed parametrized Picard-Vessiot theory  (see for example \cite{CS}, \cite{HMO}, \cite{Arr}, \cite{MS}), and this provides  new lines of research on the parametric behavior of the solutions of \eqref{eq-problem}, see Section \ref{sec-Conclusiones}.

\vskip1cm

\noindent {\it The paper is organized as folllows.} Section \ref{sec-alg-geom} presents the relation  between algebro-geometric potentials, centralizers and the $KdV$ hierarchy (proving Theorem {\bf A}). Section \ref{sec-spectral curve} establishes the field of constants of the field $K(\Gamma )$, with $\Gamma$ the spectral curve (Theorem {\bf  B} is part of Theorem \ref{thm-constants}). {\it Spectral Picard-Vessiot fields}, are defined and explicitly constructed in Section \ref{sec-Picard-Vessiot} (Theorem \ref{thm-Bron} and Theorem {\bf C}). We recover the (classical) Picard-Vessiot extensions at each value $\lambda=\lambda_0$; this is done in Section \ref{sec-specialization}, by theorems \ref{thm-classPV1} and \ref{thm-classPV2}. For rational spectral curves we prove Theorem {\bf D} in Section \ref{sec-Parametric}. The existence of a free parameter $ \tau $, allows to use symbolic integration algorithms. We illustrate the obtained  results by a family of Rosen-Morse potentials in Example \ref{example}. In this situation, $\lambda = \lambda (\tau)$, and Theorem \ref{thm-analytic}  establishes the analyticity of the common solution in a domain of $\mathbb{C}^2$, around almost every point $(x, \tau)$. Some concluding remarks are contained in Section \ref{sec-Conclusiones}.

\vskip1cm

\noindent {\bf Notation.} For concepts in differential algebra and differential Galois theory we refer the reader to \cite{CH}, \cite{VPS} or \cite{Morales}.
Let us consider algebraic variables $\lambda$ and $\mu$ with respect to $\partial$. Thus $\partial \lambda=0$ and $\partial \mu=0$ and
we can extend the derivation $\partial$ of $K$ to the polynomial ring $K[\lambda, \mu]$.
Hence $(K[\lambda, \mu],\partial)$ is a differential ring whose ring of constants is $C[\lambda, \mu]$. Given a differential commutative ring $R$ with (non trivial) derivation $\partial$,
let us denote by $R[\partial]$ the ring of differential operators with coefficients in $R$ and commutation rule $[\partial,a]=\partial a-a\partial=\partial(a)$, $a\in R$,
where $\partial a$ denotes the product in the (noncommutative) ring  $R[\partial]$ and $a'=\partial (a)$ is the image of $a$ by the derivation map.

\section{{Algebro-geometric  operators of second order }}\label{sec-alg-geom}

Let us consider a differential operator $L$ with coefficients in a differential field $(\Sigma,\partial)$, whose field of constants $C$ is algebraically closed and of characteristic zero. There are several characterizations of algebro-geometric operator, see for instance \cite{We}. We  state next what we use as the base characterization of algebro-geometric operators for this work, the Burchnall and Chaundy Theorem \cite{BC}, {adapted from \cite{We}}.
We consider the nontrivial case of operators $L\not\in C[\partial]$.

\begin{thm}\label{thm-algegeom}
Let $L$ be an order $n$ differential operator in $\Sigma[\partial]\backslash C[\partial]$. The following are equivalent:
\begin{enumerate}
    \item $L$ is an algebro-geometric operator.
    \item There exists an operator $P$ in $\Sigma[\partial]$ of order $m$, relatively prime with $n$, and a polynomial $f(\lambda,\mu)=\mu^n+R_m(\lambda)$ in $C[\lambda,\mu]$, with $R_m$ of degree $m$, such that $f(L,P)=0$.
    \item There exists an operator $P$ in $\Sigma[\partial]$ of order $m$, relatively prime with $n$, such that $[L,P]=0$.
\end{enumerate}
\end{thm}

In the case of second order differential operators we would like to highlight the structure of the centralizer $\cC(L)$ of $L$ in the ring of differential operators $\Sigma[\partial]$. 

\begin{thm}\label{thm-agCen}
Let $L$ be a second order  differential operator in $\Sigma[\partial]$. The following are equivalent:
\begin{enumerate}
    \item $L$ is an algebro-geometric operator.
    \item The centralizer of $L$ is nontrivial, $\cC(L)\neq C[L]$. More precisely, there exists {a unique monic operator $A_{2s+1}$ of minimal order $2s+1$ such that $\cC(L)=C[L,A_{2s+1}]$ and $A^2_{2s+1}+R_{2s+1}(L)=0$, with $R_{2s+1}(\lambda)$ in $C[\lambda]$ of degree $2s+1$.}
\end{enumerate}
\end{thm}
\begin{proof}
Whenever the centralizer is nontrivial, $\cC(L)\neq C[L]$, by \cite{Good}, Theorem 1.2, there exists an operator {$X_{2s+1}$} of minimal order $2s+1$ in the centralizer such that $\cC(L)$ equals the free $C[L]$-module of rank $2$ with basis $\{1, X_{2s+1}\}$, that is
\begin{equation}\label{eq-centralizer}
    \cC(L)=\{p_0(L)+p_1(L)X_{2s+1}\mid p_0,p_1\in C[L]\}=C[L]\langle 1,X_{2s+1}\rangle.
\end{equation}
Thus $\cC(L)$ equals the $C$-algebra $C[L,X_{2s+1}]$ generated by $L$ and $X_{2s+1}$. 

{In addition, any operator of the form $P=X_{2s+1} +p_0(L)$, with $p_0$ of degree $d\leq s$ is a generator of $\cC(L)=C[L,P]$. By \cite{Good}, Theorem 1.13,} 
{there exist $a,b\in C[\lambda]$ such that $X_{2s+1}^2 =a (L ) X_{2s+1} +b(L )$. Since $P^2=(X_{2s+1}+p_0(L))^2$ and $\{1, X_{2s+1}\}$ is a basis as $C(L)$ module of the centralizer, we obtain that for $p_0=-a/2$ then $P=X_2-(1/2)a(L)$ satisfies $P^2+R_{2s+1}(L)=0$ where $R_{2s+1}(\lambda)=-b(\lambda)-(1/4)a(\lambda)^2$ must be a polynomial of degree $2s+1$. Observe that $P$ is unique up to multiples $\alpha P$, $\alpha\in C$, so we choose the monic $A_{2s+1}$.} 
\end{proof}

We introduce next the terminology that will be useful to explain our contribution to the extensively studied commuting pair  $L,A_{2s+1}$.

\begin{defi}
Given an algebro-geometric operator $L$ of second order  we call the 
{monic ope\-rator $A_{2s+1}$ of minimal order $2s+1$ in $\cC(L)$ such that $\cC(L)=C[L,A_{2s+1}]$ and $A^2_{2s+1}+R_{2s+1}(L)=0$, with $R_{2s+1}(\lambda)$ in $C[\lambda]$ of degree $2s+1$},  {\sl the partner} of $L$ and say that $L$ is an {\it algebro-geometric operator of level $s$} and denote it by $L_s$.
\end{defi}

\para

Let us assume now that the algebro-geometric operator $L_s$ of level $s$ is in normal form, see \cite{We}, (2). Thus $L_s$ is a Schr\" odinger operator
\[L_s=-\partial^2+u_s, \mbox{ with }u_s\in \Sigma.\]
We prove in Appendix \ref{sec-KdV} that $u_s$ verifies one of the equations of the celebrated KdV hierarchy. This is probably a well known result, but we could not find a proof of it, so we include the proof there for completion. 

\para

{\bf Proof of Theorem A. }
{\it By Theorem \ref{thm-agCen}  and Appendix \ref{sec-KdV} the equivalences follow.}

\para

Given a Schr\" odinger operator $L_s$, to decide if it is algebro-geometric one has to look for a nontrivial commuting operator. To look for $A_{2s+1}$ one possibility is forcing the commutator of $L_s$ with an arbitrary operator of order $2n+1$ to be zero, starting with $n=1$.
The determination of the level $s$ is intrinsically related to the determination of the vector of integration constants ${\bf c}^s \in C^s$, see Theorem \ref{thm-partner}, and it is necessary for the effective computation of the partner $A_{2s+1}$ of $L_s$. The algorithmic treatment of these results can be found in \cite{MRZ}.

\para

{Observe that by Theorem \ref{thm-partner}, $A_{2s+1}$ is a differential operator with coefficients in the differential field $C\langle u_s\rangle$. Once we fix $u_s\in \Sigma$, we will work in the differential subfield   $K=C\langle u_s\rangle$ \ of $\Sigma$.}

\section{{The differential field of the spectral curve}}\label{sec-spectral curve}

Let us consider an algebro-geometric Schr\" odinger operator $L_s=-\partial^2+u_s$  of level $s$ and partner operator $A_{2s+1}$, as defined in Section \ref{sec-alg-geom}. 
{In other words, now we are fixing a KdV-potential $u_s$ in $\Sigma$ and considering commuting differential operators $L_s$ and $A_{2s+1}$ in $K[\partial]$ with $K=C\langle u_s\rangle$}. {By Theorem \ref{thm-agCen},} there exists a polynomial
\begin{equation}\label{pol-curva}
 f_s (\lambda , \mu )=\mu^2+R_{2s+1}(\lambda)
\end{equation}
in $C[\lambda,\mu]$ such that $f_s(L_s,A_{2s+1})=0$.  The constant coefficients polynomial $f_s$ is called the {\it Burchnall-Chaundy (BC) polynomial} of the pair $\{L_s,A_{2s+1}\}$, since the first result of this sort appeared in \cite{BC}. We denote by $\Gamma_s$ the affine algebraic curve in $C^2$ determined by $f_s(\lambda,\mu)=0$, which is called the {\it spectral curve of the pair} $\{L_s,A_{2s+1}\}$.

\para
{Recall that the {\it rank} of a pair of differential operators is the greatest common divisor of their orders, \cite{Wilson}. Observe that all pairs $\{L_s,A_{2s+1}\}$ studied in this paper are rank $1$ pairs. But also note that the same spectral curve could correspond to pairs of algebro-geometric operators with different rank. For instance the algebraic curve defined by $\mu^2-\lambda^3=0$ is the spectral curve of the rank $1$ pair
\[
L_1=-\partial^2+\frac{2}{x^2} \ \mbox{ and } \ A_3=\partial^3-\frac{3}{x^2}\partial+\frac{3}{x^3} \ ,
\]
in $C(x)[\partial]$ with $\partial=d/dx$, but it is also of the spectral curve of the famous pair of rank $2$ operators posted by Dixmier see \cite{PRZ},} {which moreover is a ``true rank" $2$ pair (see also \cite{PRZ} for the definition of ``true rank")
\begin{equation*}
H^{2} +2x\  \textrm{ and }  H^{2}+\frac{3}{2}(xH+Hx) \ , \mbox{ with } H=\partial^2+x^2.   
\end{equation*}}
{The case of rank $r>1$ corresponds to a vector bundle
of rank $r$ over the spectral curve.  These bundles are related with  the ``inverse'' spectral problem,  \cite{KN1, KN2}. A difficult interesting problem is to give new ``true rank" $r$ pairs. 
 Important contributions were made by Davletshina, Grinevich,  Mironov, Mokhov,   Oganesyan, Pogorelov Shamaev,   and Zheglov  (see \cite{PRZ}, and the references therein).}

\para

Traditionally BC polynomials are computed as characteristic polynomials \cite{GH}. We would like to point out that once $A_{2s+1}$ has been calculated, by Previato's Theorem (see \cite{Prev} or \cite{MRZ}, Theorem 5.4) one can compute $f_s$ by means of the differential resultant. Moreover, by \cite{Wilson} the differential resultant of $L_s-\lambda$ and $A_{2s+1}-\mu$ is related with the rank $r$ by

\begin{equation}\label{eq-sp-curve-DR}
  \dres(L_s-\lambda,A_{2s+1}-\mu)=f_s^r.  
\end{equation}
In Appendix \ref{subsec-subres} we summarize the definition and main properties of this tool.

\para

Observe that $f_s(\lambda,\mu)$ is an irreducible polynomial in $C[\lambda,\mu]$, since it has odd degree in $\lambda$ { and degree $2$ in $\mu$}. Let us denote by $(f_s)$ the prime ideal generated by $f_s$ in $C[\lambda,\mu]$ or $K[\lambda,\mu]$, abusing the notation and distinguishing them by the context. Let us consider the monomial lexicographical order with $\mu>\lambda$ in $C[\lambda,\mu]$. Given $p\in C[\lambda,\mu]$, let us denote by $p_N$ the normal form of $p$ with respect to $(f_s)$ ({that is, $ p_N $ is the remainder of dividing $ p $ by $ f_s $ in $ C [\lambda, \mu] $, }see \cite{Cox}). Observe that $p_N$ is a polynomial in $C[\lambda,\mu]$ of degree one in $\mu$. This reduction with respect to $(f_s)$ will allow us to prove the next two results.

\para

\begin{thm}\label{thm-elimination}
Let $I=(L_s-\lambda, A_{2s+1}-\mu)$ be the ideal generated by $L_s-\lambda$ and $A_{2s+1}-\mu$ in $K[\lambda,\mu][\partial]$. Let $(f_s)$ be the ideal generated by $f_s$ in $C[\lambda,\mu]$. The differential elimination ideal $I\cap C[\lambda,\mu]$ verifies
\begin{equation}\label{eq-elimIdeal}
(f_s)=I\cap C[\lambda,\mu]= \{p\in C[\lambda,\mu]\mid p(L_s,A_{2s+1})=0\}.
\end{equation}
\end{thm}
\begin{proof}
By Theorem \ref{thm-sresIdeal}, $f_s\in I$. Thus we have the next chain of inclusions of ideals in $C[\lambda,\mu]$
\[(f_s)\subseteq I\cap C[\lambda,\mu]\subseteq J=\{p\in C[\lambda,\mu]\mid p(L_s,A_{2s+1})=0\}.\]
We will prove that $J\subset (f_s)$ and therefore the equality holds.

Given $p\in J$, we can write $p=hf_s+p_N$, for $h\in C[\lambda,\mu]$. If we assume that $p_N$ is nonzero then $p_N=a(\lambda)+b(\lambda)\mu$, $a,b\in C[\lambda]$. Therefore
$a(L_s)=-b(L_s)A_{2s+1}$ which is a contradiction since $a(L_s)$ has even order and $b(L_s)A_{2s+1}$ has odd order. This proves that $p_N$ is identically zero and that $p\in (f_s)$.
\end{proof}

Let us denote by $C(\Gamma_s)$ and $K(\Gamma_s)$ the fraction fields of the domains
\begin{equation}\label{eq-domains}
C[\Gamma_s]=\frac{C[\lambda,\mu]}{(f_s)}\mbox{ and } K[\Gamma_s]=\frac{K[\lambda,\mu]}{(f_s)}
\end{equation}
respectively. Observe that $C(\Gamma_s)$ and $K(\Gamma_s)$ are usually interpreted as rational functions on {the algebraic curve $\Gamma_s$ defined by the polynomial $f_s $}. The next result gives a description of the centralizer of $L_s$ alternative to the famous one given by I. Schur \cite{Sch}, in terms of pseudodifferential operators. As a consequence the quotient field of the centralizer is a function field of one variable.

\begin{prop}\label{prop-centralizerCurve}
The centralizer $\cC_{\cD}(L_s)=C[L_s,A_{2s+1}]$ of $L_s$ in $\cD =C\langle u_s\rangle [\partial ]$  and the domain $C[\Gamma_s]$ are isomorphic commutative rings.
\end{prop}
\begin{proof}
Given $p+(f_s)$ in $C[\Gamma_s]$ it has a representative given by the normal form $p_N=a(\lambda)+b(\lambda)\mu$ of $p$ with respect to  $(f_s)$. By \eqref{eq-centralizer} we establish the isomorphism sending  $p+(f_s)$ to $a(L_s)+b(L_s)A_{2s+1}$.
\end{proof}

\para

In the history of this problem one can find different approaches that go from a local to a global treatment, "to work over the spectral curve". Note that the affine curve $\Gamma_s$ could have singular points. 

\para

One could fix a point $P_0=(\lambda_0, \mu_0)$ of $\Gamma_s$, in which case the differential operators
\begin{equation}\label{eq-localProb}
L_s-\lambda_0\mbox{ and }A_{2s+1}-\mu_0  \mbox{ belong to }K [\partial].
\end{equation}
 Hence one may even write $(\lambda_0, R_{2s+1}(\lambda_0)^{1/2})$, whenever $R_{2s+1}(\lambda_0)\neq 0$, as in \cite{GH} or \cite{Brez3}.

 In the seminal works of Krichever \cite{K77}, \cite{K}, the formal Baker-Akheizer function is given using a local parametrization $(\tau^2, \mu(\tau))$. See also \cite{Wilson}, \cite{Yagasaki}. Local parametrizations around a point $P_0$  are generally obtained as Puiseux series, see for instance \cite{SWP}, Section 2.5. They always exist in the field of Puiseux series $C\langle\langle \tau \rangle\rangle$ and the differential operators
\begin{equation}\label{eq-PuiseuxProb}
L_s-\tau^2\mbox{ and }A_{2s+1}-\mu(\tau) \mbox{ belong to }  K \langle\langle \tau \rangle\rangle [\partial].
\end{equation}
Their analytical behavior depends  on the type of point of the curve (singular or regular). Those local expansions would allow a local parametric study of the spectral problem \eqref{eq-PuiseuxProb}.

In the visionary works of Burchnall and Chaundy \cite{BC} and \cite{BC2}, the attention is driven towards the case of singular curves, for which their results regarding Abelian equations are no longer valid.
In \cite{BC2} they analyze cuspidal curves defined by $\mu^n-\lambda^m=0$, $n,m$ coprime, by means of the global parametrization $(\chi_1(\tau),\chi_2(\tau))=(\tau^m ,\tau^n )$.

\para

We propose two new approaches "to work over the spectral curve".
First, we consider $\lambda$ and $\mu$ as generic variables and assume that
\begin{equation}\label{eq-globalProb1}
L_s -\lambda \ , \ A_{2s+1} -\mu \mbox{ belong to }
K (\Gamma_s) [\partial].
\end{equation}
As operators in $K[\lambda,\mu][\partial]$, by Remark \ref{rem-gcd}, (1), they do not have a common factor, since $\dres(L_s-\lambda,A_{2s+1}-\mu)$ is nonzero, it equals the BC polynomial $f_s$. Considering them as operators in $K (\Gamma_s) [\partial]$, they have a common right factor, see Section \ref{sec-Picard-Vessiot}.

Second, in the case of rational curves, see Section \ref{sec-Parametric}, we establish the coupling governed by a global parameterization  $(\chi_1(\tau),\chi_2(\tau))$ in $C(\tau)^2$ and consider that differential operators
\begin{equation}\label{eq-globalProb}
L_s -\chi_1(\tau) \ , \ A_{2s+1} -\chi_2(\tau) \mbox{ belong to }
K (\tau) [\partial].
\end{equation}
Observe that $K(\tau)$ is a much smaller field than $K \langle\langle \tau \rangle\rangle$.
The possibility of obtaining a global parametrization depends on the genus $g$ of the curve $\Gamma_s$ and ultimately of its singular locus, see for instance \cite{BeardonNg}. If the curve $\Gamma$ is rational, $g=0$, there are symbolic algorithms to obtain a global parametrization \cite{SWP}. If $\Gamma_s$  is an elliptic curve, $g=1$, it can be parametrized  by elliptic functions in $\bbC$, \cite{Shafarevich}. For $g\geq 2$ this is a difficult open problem, some contributions have been made in this direction, for instance by Y.V. Brezhnev in \cite{Brez2}.

\vskip1cm

We can properly describe now the problem solved in this paper. We work with an algebro-geometric Schr\" odinger operator
\begin{equation}\label{eq-Problem}
    L_s-\lambda\,\,\, \mbox{ in }K(\Gamma_s)[\partial]
\end{equation}
to define, in Section \ref{sec-Picard-Vessiot}, the minimal field extension of $K(\Gamma_s)$ that contains the solutions of\break $(L_s-\lambda)\Psi=0$. In the remaning of this section we prove that $C(\Gamma_s)$ is the field of constants of $K(\Gamma_s)$.

\para

We extend the derivation $\partial$ of $K$ to the polynomial ring $K[\lambda, \mu]$ by 
\begin{equation}\label{eq-derpol}
\partial \left(\sum a_{i,j} \lambda^i\mu^j\right)=\sum \partial (a_{i,j}) \lambda^i\mu^j,\,\,\, a_{i,j}\in K
\end{equation}
with ring of constants $C[\lambda, \mu]$.
Let $(f_s)$ be the ideal generated by $f_s$ in $K[\lambda,\mu]$. Observe that  for any $p\in K[\lambda,\mu]$ it holds that
\begin{equation}
\partial (pf_s)=\partial(p) f_s+ p \partial (f_s)=\partial(p) f_s
\end{equation}
since $f_s\in C[\lambda,\mu]$. This implies that $(f_s)$ is a differential ideal in $(K[\lambda,\mu],\partial)$.
Let us consider the domains in \eqref{eq-domains} and observe that $C[\Gamma_s]\hookrightarrow K[\Gamma_s]$.
Secondly we consider the standard  differential structure of the quotient ring $K[\Gamma_s]$ given by the following:
\begin{equation}\label{eq-PartialExtended}
\tilde {\partial}(q+(f_s))=\partial(q)+(f_s), \,\,\, q\in K[\lambda,\mu].
\end{equation}
Observe that $\tilde {\partial}$ is a derivation in $K[\Gamma_s]$ since $(f_s)$ is a differential ideal. By abuse of notation we also denote by $\tilde {\partial}$ its extension to the fraction field $K(\Gamma_s)$. The next commutative diagram summarizes the situation:
\[
\xymatrix{
K[\Gamma_s]\,\, \ar@{^(->}[r] & \,\, K(\Gamma_s) \\
C[\Gamma_s]\,\,  \ar@{^(->}[u] \ar@{^(->}[r] & \,\, C(\Gamma_s) \ar@{^(->}[u]
}
\]

\para

To work in the ring $K[\Gamma_s]$ we will consider special representatives of its elements fixing the monomial lexicographical order with $\mu>\lambda$ in $K[\lambda, \mu]$. Given $p+(f_s)\in K[\Gamma_s]$, with representative $p\in K[\lambda,\mu]$, let us denote by $p_N$ the normal form of $p$ with respect to $(f_s)$ (see \cite{Cox})
. We will call $p_N$ the {\it normal form of $p$ on $\Gamma_s$}. Observe that $p_N$ is a polynomial in $K[\lambda,\mu]$ of degree one in $\mu$. The following observations will be very important in what follows.
\begin{rem}\label{rem-factorLambda}
Given $q$ and $h$ polynomials in $K[\lambda,\mu]$ of degree one in $\mu$.
\begin{enumerate}
\item  If $h$ is a factor of $q$ then $q=\beta h$ for some nonzero $\beta$ in $K[\lambda]$.

\item  $q=a+b\mu$ is irreducible in $K[\lambda,\mu]$ if and only if $\gcd(a,b)=1$.

\item Given $q$ of degree one in $\mu$, we can factor it as $q=\Lambda \hat{q}$ where $\Lambda \in K[\lambda]$ and $\hat{q}$ is irreducible of degree one in $\mu$.
\end{enumerate}
\end{rem}

\begin{prop}\label{lem-unit}
Let $q$ be a polynomial in $K[\lambda,\mu]$ of degree greater or equal than one in $\mu$.
\begin{enumerate}
\item Let $q+(f_s)$ be a nonzero element in $K[\Gamma_s]$. There exists {a nonzero } $T\in K[\lambda,\mu]$ and a nonzero $\Lambda\in K[\lambda]$ such that $Tq+(f_s)=\Lambda+(f_s)$.\label{item-lem-unit}

\item Every nonzero $\frac{p}{q}$ in $K(\Gamma_s)$ equals $\frac{r}{\Lambda}$ in $K(\Gamma_s)$ where $r\in K[\lambda,\mu]$ and $\Lambda\in K[\lambda]$.
\end{enumerate}
\end{prop}
\begin{proof}
The polynomials $q$ and $f_s$ can be seen as polynomials in $\mu$ with coefficients in the field $K(\lambda)$. By hypothesis, $f_s$ does not divide $q$ in $K[\lambda,\mu]$. This implies that $f_s$ does not divide $q$ in $K(\lambda)[\mu]$, { since $f_s$  is monic by \eqref{eq-sp-curve-DR};} therefore they are coprime. There exist $A,B\in K(\lambda)[\mu]$ such that $Aq+Bf_s=1$, see \cite{Cox}, Chapter 1, \S 5, Proposition 6. We can write $A(\lambda,\mu)=g(\lambda,\mu)/a(\lambda)$ and $B(\lambda,\mu)=h(\lambda,\mu)/b(\lambda)$, with $g,h\in K[\lambda,\mu]$ and nonzero $a,b\in K[\lambda]$. Thus $bgq+ahf_s=ab$ in $K[\lambda,\mu]$. In $K[\Gamma_s]$ we have
\begin{equation}\label{eq-unit}
(T+(f_s))(q+(f_s))=\Lambda(\lambda)+(f_s),\mbox{ with }T=bg\mbox{ and }\Lambda=ab.
\end{equation}
Given a nonzero $p+(f_s)$ in $K[\Gamma_s]$ from \eqref{eq-unit} we have $Tqp+(f_s)=\Lambda p+(f_s)$ thus $\frac{p}{q}$ equals $\frac{Tq}{\Lambda}$ in $K(\Gamma_s)$ and for $r=Tq$ statement 2 follows.
\end{proof}

{We are now ready to prove Theorem B.}

\begin{thm}\label{thm-constants} 
\begin{enumerate}
\item The ring of constants of $(K[\Gamma_s],\tilde{\partial})$ is $C[\Gamma_s]$.
\item The field of constants of $(K(\Gamma_s),\tilde{\partial})$ is $C(\Gamma_s)$.
\end{enumerate}
\end{thm}
\begin{proof}
Let us consider $p+(f_s)$ in $K[\Gamma_s]$ such that $\tilde{\partial} (p+(f_s))=0$. Let $p_N=a+b\mu$ be its normal form on $\Gamma_s$, then $\partial (a)+\partial (b)\mu \in (f_s)$ in $K[\lambda,\mu]$. Then $\partial(a)=\partial(b)=0$. Hence $p+(f_s)$ belongs to $C[\Gamma_s]$, which proves statement 1.

Let us also consider $v\in K(\Gamma_s)$ such that $\tilde{\partial}(v)=0$. By Proposition \ref{lem-unit}, 2, we have
$v=\frac{p}{\Lambda_1}$ in $K(\Gamma_s)$, with $p\in K[\lambda,\mu]$ and $\Lambda_1\in K[\lambda]$.

If $p=\Lambda_2\in K[\lambda]$ then
$\tilde{\partial}(v)=\frac{H}{\Lambda_1^2}$ in $K(\Gamma_s)$  with $H=\partial(\Lambda_2)\Lambda_1-\partial(\Lambda_1)\Lambda_2.$
Thus $0=\Lambda_1^2\tilde{\partial}(v)=H+(f_s)$. Since $H$ is a polynomial in $K[\lambda]$ then $H=0$. Hence $\Lambda_2=\gamma \Lambda_1$, with $\gamma$ in the field  of constants $C(\lambda)$ of $K(\lambda)$, and $v=\gamma$ in $K(\Gamma_s)$. Consequently $v\in C(\Gamma_s)$.

If $p\notin K[\lambda]$, by Remark \ref{rem-factorLambda} we can write $v=\frac{p_N}{\Lambda_1}$ in $K(\Gamma_s)$, where the normal form on $\Gamma_s$ of $p$ equals $p_N=a+b\mu$. Now $0=\Lambda_1^2\tilde{\partial}(v)=H+(f_s)$, with
\[
H=\Lambda_1\partial(p_N)-\partial(\Lambda_1)p_N.
\]
Since the degree in $\mu$ of $H$ equals one then $H=0$ in $K[\lambda,\mu]$. This implies that $p_N$ divides $\partial(p_N)$ and by Remark \ref{rem-factorLambda}, $\partial(p_N)=\beta p_N$, with $\beta\in K[\lambda]$. Thus $a$ and $b$ are solutions of the linear differential equation $\partial (\Psi)=\beta(\lambda) \Psi$ then $a=c b$, $c=c_1/c_2$ with $c_1,c_2\in C[\lambda]$. Therefore $v=\frac{b(c_2+\mu)}{c_1\lambda_1}$ in $K(\Gamma_s)$ and
\[0=\tilde{\partial}(v)=(c_1+\mu)\tilde{\partial}(w)\Longleftrightarrow \tilde{\partial}(w)=0,\mbox{ with }w=\frac{b}{c_1\Lambda_1}\mbox{ in }K(\Gamma_s).\]
Thus  $w\in C(\Gamma_s)$, which proves that $v\in C(\Gamma_s)$, and statement 2 is proved.
\end{proof}

\section{{Spectral Picard-Vessiot fields }}\label{sec-Picard-Vessiot}

We are ready now to introduce the main concept of this paper, the {\it spectral Picard-Vessiot (PV) field} of the equation
\begin{equation}\label{eq-problem4}
(L_s-\lambda)\Psi=0,
\end{equation}
where, as before, $L_s =-\partial^2+u_s$ is an algebro-geometric Schr\" odinger operator of level $s$. We consider
 $(L_s-\lambda)(\Psi)=0$ as an homogeneous linear differential equation of second order  with coefficients in $(K(\Gamma_s),\tilde{\partial})$. Since $\tilde{\partial}$ extends the derivation $\partial$ of $K$,  as can be deduced from \eqref{eq-PartialExtended}, when there is no room for confusion we write $\partial $ instead of $\tilde{\partial}$.

\para 

Let us  recall the definition of Picard-Vessiot extension following Kaplansky (see \cite{Ka}, \cite{VPS} or \cite{CH} for instance).

{\begin{defi}\label{def-Kaplanski-PV}
Let $y^{(n)}+a_{n-1}y^{(n-1)}+\cdots +a_1 y' +a_0 y =0$ be a linear homogeneous differential equation with coefficients in the differential field $\Sigma$.
We say that a differential field $\cE$ containing $\Sigma$ is a {\it Picard-Vessiot extension of $\Sigma$ for the above equation}, if the following conditions are satisfied:
\begin{enumerate}
\item $\cE=\Sigma <u_1 ,\dots ,u_n >$,  the differential field extension of $\Sigma$ generated by a fundamental set of solutions $\{u_1 ,\dots ,u_n \}$ of $y^{(n)}+a_{n-1}y^{(n-1)}+\cdots +a_1 y' +a_0 y =0$.
 \item $\cE$ and $\Sigma$ have the same field of constants.
\end{enumerate}
\end{defi}}

In this paper we use Picerd-Vessiot extensions as in Definition \ref{def-Kaplanski-PV}, but, since $L_s$ is an algebro-geometric operator we will be able to give a precise description of its Picard-Vesiot field in connection with its  spectral curve. This particular field structure we will call {\it spectral Picard-Vessiot  field over the curve}. { Due to its importance for this paper we present next the definition.}

\para 

\begin{defi}\label{def-PV}
A differential field extension $\cE$ of $K(\Gamma_s)$ is called a {\it spectral Picard-Vessiot  field over the curve} $\Gamma_s$ of the equation $(L_s-\lambda)(\Psi)=0$ if the following conditions are satisfied:
\begin{enumerate}
\item $\cE=K(\Gamma_s)\langle \Psi_1,\Psi_2\rangle$,  the differential field extension of $K(\Gamma_s)$ generated by
 $\Psi_1,\Psi_2$, where $\{\Psi_1,\Psi_2\}$ is a fundamental set of solutions of $(L_s-\lambda)(\Psi)=0$.

\item $\cE$ and $K(\Gamma_s)$ have the same field of constants $C(\Gamma_s)$.
\end{enumerate}
\end{defi}

\para

{ Afterwards, we will prove its existence  highlighting the importance of its field of constants $C(\Gamma_s)$ and proving that it is a transcendental Liouvillian extension of $ K(\Gamma_s)$. }

\para 

{For the convenience of the reader, we recall some definitions from differential algebra.  Let $K$ be a differential field and $\Sigma$ a differenttial extension of $K$. An element $t$ in $\Sigma$ is a {\it primitive over $K$ }if $Dt\in K$; the element $t$ is an {\it hyperexponential over $K$} if $Dt/t \in K$; and $t$ is called {\it Liouvillian element of $\Sigma$} if $t$ is either algebraic, or a primitive or an hyperexponential over $K$. The field $\Sigma$ is a  {\it  Liouvillian extension of $K$} if  there are $t_1 ,\dots ,t_n$ in $\Sigma$  such that $\Sigma = K(  t_1 ,\dots ,t_n )  $ and $t_i$
is Liouvillian over $K(t_1 ,\dots ,t_{i-1})$ for $i=1, \dots , n$.}

\para

{The next proposition shows that, as an element of $K(\Gamma_s)[\partial]$, the algebro-geometric operator $L_s-\lambda$ has an intrinsic order one (right) factor,}  {that is linked to the (unique) partner $A_{2s+1}$ of $L_s$.}

\para

\begin{prop}\label{thm-L1}
The monic greatest common (right) divisor of the differential operators $L_s-\lambda$ and $A_{2s+1}-\mu$ in $K(\Gamma_s)[\partial]$ is the order one operator
\[\partial-\phi_s,\mbox{ where }
\phi_s=\frac{\mu+\alpha(\lambda)}{\varphi(\lambda)},\]
for $\alpha$ and $\varphi$  nonzero polynomials in $K[\lambda]$. Moreover $\phi_s$ is nonzero in $K(\Gamma_s)$ and {we call $\partial-\phi_s $ {\sl the intrinsic right factor of} $L_s -\lambda$}.
\end{prop}
\begin{proof}
As defined in Appendix \ref{subsec-subres}, let us consider the differential resultant $G_0$ and first subresultant $G_1$ of $L_s-\lambda$ and $A_{2s+1}-\mu$ in \eqref{eq-G0} and \eqref{eq-G1} respectively.
Since $G_0$ is zero in $K(\Gamma_s)$ by Theorem \ref{thm-subres} the greatest common factor of $L_s-\lambda$ and $A_{2s+1}-\mu$ is nontrivial. In addition $G_1$ is an order one differential operator
\[G_1=(-\mu-\alpha(\lambda))+\varphi(\lambda) \partial\]
where $\alpha$ and $\varphi$ are nonzero polynomials in $K[\lambda]$ and  by Theorem \ref{thm-subres} it is the greatest common right divisor.

We will also write $\phi_s$ to denote the element $\phi_s$ in $K(\Gamma_s)$.
Observe that $\phi_s=0$ in $K(\Gamma_s)$ if and only if $
\mu+\alpha+(f_s)=0$ in $K[\Gamma_s]$. But this is not possible since $f_s$, which has degree $2$ in $\mu$, is not a factor of $\mu+\alpha$ in $K[\lambda,\mu]$. This proves the last claim. 
\end{proof}

\begin{rem} We proved Proposition \ref{thm-L1} using the Differential Subresultant Theorem \ref{thm-subres}. In addition the first subresultant $G_1=\varphi_2 \partial+\varphi_1$, see \eqref{eq-G1}, of $L_s-\lambda$ and $A_{2s+1}-\mu$ can be used to compute the factor $\partial-\phi_s$, see \cite{MRZ} for an algorithmic approach.
\end{rem}

We have the next factorization over the spectral curve
\begin{equation}\label{eq-factKGamma}
    L_s-\lambda=(-\partial-\phi_s)(\partial-\phi_s),\mbox{ in }K(\Gamma_s)[\partial],
\end{equation}
{since, once this right factor is set, the only possibility as a left factor  is $-\partial-\phi_s$.}
Let us define $\phi_+:=\phi_s$ before obtaining another factorization of $L_s-\lambda$. Observe that the BC-polynomial of $L_s-\lambda$ and  $A_{2s+1}+\mu$ is also  $f_s (\lambda,-\mu)=f_s (\lambda,\mu)$.
Applying Proposition \ref{thm-L1},  we obtain another factorization of $L_s-\lambda$, namely
\begin{equation}
L_s-\lambda=(-\partial-\phi_-)(\partial-\phi_-),\mbox{ in }K(\Gamma_s)[\partial],
\end{equation}
with
\[
\phi_-=\frac{-\mu+\alpha(\lambda)}{\varphi(\lambda)}.
\]
Both $\phi_+ $ and $\phi_- $ are solutions of the Riccati equation $ \partial (\phi)+\phi^2=u_s +\lambda$ with coefficients in $K(\Gamma_s )$.

Nonzero solutions $\Psi_+$ and $\Psi_-$ of $(L_s-\lambda)(\Psi)=0$ are defined respectively by the differential relations
\begin{equation}\label{eq-Psi}
\partial(\Psi_+)=\phi_+\Psi_+\mbox{ and }\partial(\Psi_-)=\phi_-\Psi_-
\end{equation}
and hence $\Psi_+$ and $\Psi_-$ belong to the differential closure of the field $K(\Gamma_s)$, \cite{Kolchin}.
Therefore
\[\frac{\partial (\Psi_+)}{\Psi_+}=\phi_+\mbox{ and } \frac{\partial (\Psi_-)}{\Psi_-}=\phi_-\]
belong to $K(\Gamma_s)$ and $K(\Gamma_s)\langle \Psi_+\rangle$ and $K(\Gamma_s)\langle \Psi_-\rangle$ are Liouvillian extensions
of $K(\Gamma_s)$.

\begin{lem}\label{lem-sol}
Given $\Psi_+$ and $\Psi_-$ as in \eqref{eq-Psi}, it holds that:
\begin{enumerate}[a.]
\item $\{\Psi_+,\Psi_-\}$ is a fundamental set of solutions of $(L_s-\lambda)(\Psi)=0$.
\item $\Psi_+\Psi_-\in K(\Gamma_s)$.\label{formula-ssp}
\end{enumerate}
\end{lem}
\begin{proof}
Trivially $\Psi_+$ and $\Psi_-$ are nonzero solutions of $L_s-\lambda$. We will prove that their wronskian is nonzero. Observe that $\mu$ is a constant in $K(\Gamma_s)$ {, that is $\mu\in C(\Gamma_s)$ and furthermore it is nonzero, and the operator $L_s-\lambda$ is in normal form. }
Since $\partial (w(\Psi_+,\Psi_-))=0$, $w(\Psi_+,\Psi_-)$ belongs to $C(\Gamma_s)$. The following computation
\begin{equation}\label{eq-tensor1}
\frac{w(\Psi_+,\Psi_-)}{\Psi_+\Psi_-}=\frac{\partial(\Psi_+)}{\Psi_+}-\frac{\partial(\Psi_-)}{\Psi_-}=\phi_+-\phi_-=
\frac{2\mu}{\varphi} 
\end{equation}
implies that $w(\Psi_+,\Psi_-)\neq 0$ in $C(\Gamma_s)$. This formula implies that
\begin{equation}\label{eq-tensor2}
\Psi_+\Psi_-=\frac{\varphi w(\Psi_+,\Psi_-)}{2\mu}\in K(\Gamma_s),
\end{equation}
which completes the proof.
\end{proof}

\begin{rem}
{Notice that  Lemma \ref{lem-sol}, \ref{formula-ssp}  establishes that the product of two solutions belongs to the field of coefficients of the Schr\"odinger operator. Thus, the Picard-Vessiot structures $K(\Gamma_s)\langle \Psi_\pm\rangle$ will also benefit from this fact, analogously to the classical case, whenever a particular value $\lambda=\lambda_0$ is chosen. In this case, since Hermite \cite{Her}, Halphen \cite{Hal}, Drach \cite{Drach2, Drach3}, the study of the second symmetric power of the operator is fundamental to analyze the Lam\'e equation (see Wittaker-Watsom \cite{WW}, page 570 and the references therein).
}

{We would like to point out that the function of Lemma \ref{lem-sol}, \ref{formula-ssp}  was also used in the important work of Gelfand and Dickii to study the asymptotic behavior of the resolvent, see \cite{GD}.
}
\end{rem}

By Lemma \ref{lem-sol}, \ref{formula-ssp} the next equality of differential filed extensions of $K(\Gamma_s)$ holds
\begin{equation}\label{eq-decomp}
K(\Gamma_s)\langle \Psi_+,\Psi_{-}\rangle=K(\Gamma_s)\langle \Psi_+\rangle=K(\Gamma_s)\langle \Psi_-\rangle.
\end{equation}


Let us denote by $\Psi_s=\Psi_+$, which is defined
 by the differential relation $\partial(\Psi_s)=\phi_s\Psi_s$. We will prove that the Liouvillian extension
\begin{equation}\label{eq-extLliouvi}
    K(\Gamma_s)\subset K(\Gamma_s)\langle \Psi_s\rangle
\end{equation}
is transcendental, showing that  $\Psi_s$ is a transcendental element over $K(\Gamma_s)$.
In fact, this is intrinsically related to the determination of the subfield of constants of the field $K(\Gamma_s)\langle \Psi_s\rangle$. For this purpose we will use results from the book of  M. Bronstein (see \cite{Bron} and the reference therein). {We quote here some results from \cite{Bron} for the convenience of the reader.}

\begin{defi}[\cite{Bron}, Definition 3.4.3]
Let $(F,\delta)$ be a differential field. We say that $\phi\in F$ is a {\rm logarithmic derivative of a $F$-radical} if there exist a nonzero $v$ in $F$ and an integer $n\neq 0$ such that $n\phi=\delta v/v$.
\end{defi}

\begin{thm}[\cite{Bron}, Theorem  5.1.2]\label{thm-Bronstein}
{If $t$ is an hyperexponential over $F$ and $Dt/t$ is not a logarithmic derivative of a $F$-radical, then $t$ is a monomial over $F$, the field of constants of $F(t)$ equals the field of constants of $F$; and $\cS=F$. Conversely, if $t$ is transcendental and hyperexponential
over $F$, and the field of constants of $F(t)$ equals the field of constants of $F$ then $Dt/t$ is not a logarithmic derivative of a $F$-radical.}
\end{thm}

\begin{thm}[\cite{Bron},   Corollary 3.3.1.]\label{corollary-Bronstein}
 { Let $(F,D)$ be a differential field and let $E$ be a separable algebraic
extension of $F$. Let also $\bf C$ be the constant field of $F$ with respect to the derivation $D$ and let $\overline{{\bf C}}^E$ be the algebraic closure of
$\bf C$ in $E$, i.e. the subfield of all the elements of $E$ that are algebraic over $\bf C$.
Then $D$ can be extended uniquely to $E$, call it $\tilde{D}$, and the field of constants of $E$ with respect to  $\tilde{D}$ is also $\overline{C}^E$. In addition, if $E$ is algebraically closed, then field of constants of $E$ with respect to $\tilde{D}$ is an algebraic closure of $\bf C$.}
\end{thm}

{The next diagram corresponds to Theorem \ref{corollary-Bronstein}.}

$$ \begin{tikzcd}
\overline{{\bf C}}^E \rar[hookrightarrow]  \dar[dash]  & (E,\tilde{D}) \dar[dash] \\%
{\bf C} \rar[hookrightarrow]& (F,D)
\end{tikzcd}
$$

Observe that if there exists a nonzero $v\in  K(\Gamma_s)$ and a nonzero integer $n$ such that
\[\frac{\tilde{\partial}(v)}{nv}=\phi_s=\frac{\tilde{\partial}( \Psi_s)}{ \Psi_s}\]
then for $c= \Psi_s^n/v$ we have $\tilde{\partial}(c)=0$ and also $ \Psi_s^n-cv=0$. This ensures that $ \Psi_s$ is algebraic over a differential field that is generated by $K(\Gamma_s)$ and a possibly new constant $c$. We will prove that this is not the case for $ \Psi_s$.

\begin{lem}\label{prop-kradical}
There does not exists a nonzero $v\in K(\Gamma_s)$ such that $\phi_s=\frac{\tilde{\partial} (v)}{nv}$ for a nonzero integer $n$. That is, $\phi_s$ is not a logarithmic derivative of a $K(\Gamma_s)$-radical.
\end{lem}
\begin{proof}
Let us assume that there exists $v\in K(\Gamma_s)$, $v\neq 0$, $\tilde{\partial}(v)\neq 0$ such that $\phi_s=\frac{\tilde{\partial} (v)}{nv}$ for a nonzero integer $n$. By Proposition \ref{lem-unit}, 2, we can write $v=\frac{p}{\Lambda_1}$ in $K(\Gamma_s)$ for some nonzero $p\in K[\lambda,\mu]$ and $\Lambda_1\in K[\lambda]$. Let $p_N$ be the normal form  of $p$ on $\Gamma_s$.

Let us consider the polynomial in $K[\lambda,\mu]$
\begin{equation}\label{eq-H}
H=\left[\partial (p_N)\Lambda_1-p_N\partial(\Lambda_1)\right]\varphi-np_N\Lambda_1 (\mu+\alpha).
\end{equation}
Recall that $\phi_s=\frac{\mu+\alpha}{\varphi}$ as in Theorem \ref{thm-L1}, hence $\alpha,\varphi\in K[\lambda]$ and
\[\frac{\tilde{\partial} (v)}{nv}-\phi_s=\frac{H}{np_N\Lambda_1\varphi}=0\,\,\,\mbox{ in }K(\Gamma_s).\]
Now we apply Proposition \ref{lem-unit}, 1, for $q=np_N\Lambda_1\varphi$. Then there exists {nonzero} $T\in K[\lambda,\mu]$ and a nonzero $\Lambda_2\in K[\lambda]$ such that
\[0=\frac{\Lambda_2H}{np_N\Lambda_1\varphi}=\frac{\Lambda_2H}{q}=\frac{TH}{1}\mbox{ in }K(\Gamma_s).\]
Therefore in $K(\lambda,\mu)$
\[\frac{\Lambda_2 H}{np_N\Lambda_1\varphi}=N f_s,\mbox{ for some }N\in K[\lambda,\mu].\]
Finally we obtain
\begin{equation}\label{eq-H2}
\Lambda_2\left[\left[\partial (p_N)\Lambda_1-p_N\partial(\Lambda_1)\right]\varphi-np_N\Lambda_1 (\mu+\alpha)\right]=np_N\Lambda_1\varphi N f_s.
\end{equation}

If $p_N\in K[\lambda]$ then the degree in $\mu$ of the LHS of \eqref{eq-H2} is $1$ and of RHS of \eqref{eq-H2} is at least $2$. Thus this is not possible.
We have proved that $v$ cannot be equal to $\gamma$ in $K(\Gamma_s)$ , with $\gamma\in K(\lambda)$. Hence it remains to check the case where $p_N$ is not in $K[\lambda]$.

Let us assume that $p_N\notin K[\lambda]$. By \eqref{eq-H2}, $p_N$ is a factor of $\Lambda_1\Lambda_2\partial(p_N)\varphi$. Then by Remark \ref{rem-factorLambda}, 1,  $\partial(p_N)=\Lambda_3 p_N$, with $\Lambda_3\in K[\lambda]$. Hence equality \eqref{eq-H2} becomes
\begin{equation}\label{eq-H3}
\Lambda_2\left[\left[\Lambda_3 \Lambda_1-\partial(\Lambda_1)\right]\varphi-n\Lambda_1 (\mu+\alpha)\right]=n\lambda_1\varphi N f_s.
\end{equation}
Observe that the degree in $\mu$ of the LHS of \eqref{eq-H3} is $1$ and of RHS of \eqref{eq-H3} is at least $2$. But this is a contradiction.
Therefore we conclude that such $v$ does not exist, which proves the result.
\end{proof}

It is an immediate consequence of the results in  \cite{Bron} that the previous lemma is equivalent to the next theorem.

\begin{thm}\label{thm-Bron}
Let $L_s$ be an algebro-geometric Schr\" odinger operator with spectral curve $\Gamma_s$. Let us consider the intrinsic right factor of $L_s-\lambda$  from Proposition \ref{thm-L1}, \ $\partial-\phi_s$, \  in $K(\Gamma_s)[\partial]$. A nonzero solution $\Psi_s$ of $(L_s-\lambda)\Psi =0$
defined by the differential relation $\partial(\Psi_s)=\phi_s\Psi_s$
is transcendental over $K(\Gamma_s)$ and the field of constants of $K(\Gamma_s)\langle \Psi_s\rangle$ equals the field of constants of $K(\Gamma_s)$.
\end{thm}
\begin{proof}  By Lemma \ref{prop-kradical},
applying Theorem \ref{thm-Bronstein} (\cite{Bron}, Theorem 5.1.2) to the hyperexponential $t=\Psi_s$ and the differential field $(K(\Gamma_s),\tilde{\partial})$, the result follows.
\end{proof}

We then proved the existence of the spectral Picard-Vessiot field over the curve $\Gamma_s$ of the equation $(L_s-\lambda)(\Psi)=0$.

\para

{
{\bf Proof of Theorem C.}
{\it Theorem {\bf C} is a direct consequence of Lemma \ref{lem-sol}, Theorem \ref{thm-Bron} and Definition \ref{def-PV}.}
}

\para

We illustrate {Theorem {\bf C}} 
with the following commutative diagram, whose second row shows the fields of constants:

\[
\xymatrix{
K(\Gamma_s)\,\, \ar@{^(->}[r] & \,\, K(\Gamma_s)\langle \Psi_s\rangle \\
C(\Gamma_s)\,\,  \ar@{^-->}[u] \ar@{^==}[r] & \,\, C(\Gamma_s)\ar@{^(->}[u]
}.
\]

In the next two sections we show applications of this new structure,  the spectral PV field of the operator $L_s-\lambda$.

\section{{Classical Picard-Vessiot fields}}\label{sec-specialization}

Let $L_s=-\partial^2+u_s$ be an algebro-geometric Schr\" odinger operator with spectral curve $\Gamma_s$ and nonconstant potential $u_s$. For any fixed $P_0=(\lambda_0,\mu_0)$ in $\Gamma_s$, we will describe the Picard-Vessiot (PV) extension over $K=C\langle u_s\rangle$ of $L_s-\lambda_0$.  

We recall that, when a particular value of the spectral parameter is fixed,   in the definition of the  PV extension for $L_s-\lambda_0$, the field of constants used to be an algebraically closed field $C$ of characteristic $0$, see \cite{VPS} and also \cite{AMW}. Then, now we are looking for a differential field $\Sigma=K\langle y_1,y_2\rangle$,  the differential field extension of $K$ generated by
 $y_1,y_2$, where $\{y_1,y_2\}$ is a fundamental set of solutions of $(L_s-\lambda_0)(y)=0$, whose field of constants is also $C$. We will call $\Sigma $ {\it classical Picard-Vessiot Extension} in this case.

\para

For a nonsingular point $P_0$ of $\Gamma_s$, the dimension of the space of common solutions of $L_s-\lambda_0$ and $A_{2s+1}-\mu_0$ is known to be one, \cite{Wilson}, Theorem 5.8. In the next proposition we prove that this holds even for singular points.

\begin{prop}\label{prop-spefact}
For every $P_0=(\lambda_0,\mu_0)$ in $\Gamma_s$ the differential operators $L_s-\lambda_0$ and $A_{2s+1}-\mu_0$ have a greatest common right factor of order one $\partial-\phi_0$ in $K[\partial]$ with nonzero
\begin{equation}\label{eq-phi0}
\phi_0=\frac{\mu_0+\alpha(\lambda_0)}{\varphi(\lambda_0)}\in K.
\end{equation}
\end{prop}
\begin{proof}
By the Subresultant Theorem \ref{thm-subres}, $L_s-\lambda_0$ and $A_{2s+1}-\mu_0$ have a common factor  {  of order one in $K[\partial ]$}, namely $\varphi_1(P_0)+\varphi_2(P_0) \partial$ the specialization of \eqref{eq-G1} to $P_0=(\lambda_0,\mu_0)$, { hence $\varphi_2(\lambda_0)\neq 0$}
\end{proof}

For a fixed $P_0=(\lambda_0,\mu_0)$ (singular or not) of $\Gamma_s$, we have the factorization
\begin{equation}\label{eq-L0fac}
    L_s-\lambda_0=(-\partial-\phi_0)(\partial-\phi_0),\,\,\mbox{ in }K[\partial].
\end{equation}
Let us denote by $y_0$ a nonzero solution of $(\partial-\phi_0)(y)=0$. Thus $y_0$ is a common solution of
\begin{equation}\label{eq-systemP0}
\left\{\begin{array}{c}
     (L_s-\lambda_0)(y)=0,\\
     (A_{2s+1}-\mu_0)(y)=0.
\end{array}\right.
\end{equation}
For $u_s$ an analytic function in a suitable complex domain, $y_0$ is the stationary Baker-Akhiezer function
\[
y_0=\Psi(P_0,x,x_0)=\exp \left(\int_{x_0}^{x}{\phi_s(P_0,x')dx'} \right), \,\,\, P_0\in\Gamma_s .
\]
Traditionally the   Baker-Akhiezer function is only defined for nonsingular affine points of $\Gamma_s$ (\cite{GH}, (1.41)) .

\begin{prop}\label{thm-PVconst}
For any $P_0$ in $\Gamma_s$, let us consider
a nonzero solution $y_0$ {in a differential closure of the differential field $K$ } of  $(\partial-\phi_0) y=0$.
The field of constants of the differential field $K\langle y_0 \rangle$ is $C$.
\end{prop}
\begin{proof}
Being defined by $y'_0/y_0=\phi_0\in K$, we know that $K\langle y_0 \rangle$ is a Liouvillian extension of $K$. If $y_0$ is algebraic, the result follows {by Theorem \ref{corollary-Bronstein} (\cite{Bron}, Corollary 3.3.1)}. If $y_0$ is transcendental the result follows by \cite{Bron}, theorems 5.1.1 and 5.1.2.
\end{proof}

To describe the PV extension of $L_s-\lambda_0$, we must distinguish two different types of point $P_0=(\lambda_0,\mu_0)$ in the curve, the ones with $\mu_0\neq 0$ and those with $\mu_0= 0$, that is
the finite set
\begin{equation}\label{eq-Zs}
Z_s=\Gamma_s\cap (C\times\{0\})=\{(\lambda,0)\mid R_{2s+1}(\lambda)=0\}.    
\end{equation}

Observe that $Z_s$ contains all the affine singular points and ramification points of $\Gamma_s$.

\begin{thm}\label{thm-classPV1}
Let us fix $P_0=(\lambda_0,\mu_0)$ in $\Gamma_s\backslash Z_s$. The PV extension of the equation $(L_s-\lambda_0)y=0$, is the Liouvillian extension $K\langle y_0 \rangle$ of $K$ by a nonzero solution $y_0$ of $(\partial-\phi_0) y=0$ as in \eqref{eq-L0fac}.
\end{thm}
\begin{proof}
Applying Proposition  \ref{prop-spefact} to the point $P_0'=(\lambda_0,-\mu_0)$ gives a new factorization
\[L_s-\lambda_0=(-\partial-\phi_{0-})(\partial-\phi_{0-}),\,\, \mbox{ in }K[\partial].\]
Then we have
\begin{equation}\label{eq-phi+0}
\phi_{0+}=\phi_0=\frac{\mu_0+\alpha(\lambda_0)}{\varphi(\lambda_0)}\mbox{ and }\phi_{0-}=\frac{-\mu_0+\alpha(\lambda_0)}{\varphi(\lambda_0)}.
\end{equation}
Hence we consider nonzero solutions $y_+$ and $y_{-}$  of the differential equations
$\partial(y)=\phi_{0+}y$ and $\partial(y)=\phi_{0-}y$ respectively, in a differential closure of $K$.
The equality
\[\frac{w(y_{+},y_{-})}{y_{+}y_{-}}=\phi_{0+}-\phi_{0-}=
\frac{2}{\varphi(\lambda_0)}\mu_0\neq 0\]
implies that $W_0=w(y_{+},y_{-})\neq 0$ in $C$. Therefore $\{y_{+},y_{-}\}$ is a fundamental set of solutions of $(L_s-\lambda_0)(y)=0$. Moreover
\[y_{+}y_{-}=\frac{\varphi(\lambda_0)W_0}{2\mu_0}\in K,\]
hence $K\langle y_{+},y_{-}\rangle=K\langle y_{+}\rangle$.
In addition, by Proposition \ref{thm-PVconst},  $K\langle y_{+}\rangle$ and $K$ have the same field of constants $C$, which proves that $K\langle y_{+}\rangle$ is the PV field of $L_s-\lambda_0$.
\end{proof}

Observe that $(\partial-\phi_0)(y_0)=0$ implies $(-\partial-\phi_0)(y_0^{-1})=0$. Thus a solution of the Risch differential equation
\begin{equation}\label{eq-Risch}
    \partial(y)-\phi_0 y=y_0^{-1}
\end{equation}
over the differential field $K\langle y_0\rangle$, see \cite{Bron}, Section 6, would be a solution of $(L_s-\lambda_0)(y)=0$ because of the factorization \eqref{eq-L0fac}.

\begin{thm}\label{thm-classPV2}
Let us fix $(\lambda_0,0)$ in $Z_s$. Given nonzero solutions $y_0$ of $(\partial-\phi_0) y=0$ as in \eqref{eq-L0fac} and $y_1$ of \eqref{eq-Risch}, we have a chain of Liouvillian extensions
\begin{equation}\label{thm-especializado_0}
 K\subset K\langle y_0\rangle \subset K\langle y_0, y_1\rangle   
\end{equation}
with field of constants $C$, and then $K\langle y_0, y_1\rangle$ is the PV extension of $(L_s-\lambda_0)y=0$.
\end{thm}
\begin{proof}
By Proposition \ref{thm-PVconst} the field of constants of $K\langle y_0\rangle$ is $C$.
Let us assume that $y_1\notin K\langle y_0\rangle$.
We have a fundamental system of solutions $\{y_0,y_1\}$ of $(L_s-\lambda_0)(y)=0$ since
\[\frac{w(y_0,y_1)}{y_0y_1}=\frac{\partial (y_0)}{y_0}-\frac{\partial (y_1)}{y_1}=
\frac{-1}{y_0y_1}.\]
Remains  to prove that the field of constants of $K\langle y_0, y_1\rangle$ is $C$.

Let $\hat{y}_1$ be a nonzero solution of $\partial-y_0^{-1}$. Without loss of generality, we can take $y_1=y_0\hat{y}_1$ as the solution of \eqref{eq-Risch}. Observe that $K\langle y_0, y_1\rangle=K\langle y_0,\hat{y}_1\rangle$ is a Liouvillian extension of $K\langle y_0\rangle$ and $\hat{y}_1$ is a primitive element over $K\langle y_0\rangle$. If $\partial (\hat{y}_1 )$ is a not the derivative of an element of $K\langle y_0\rangle$, by \cite{Bron}, Theorem 5.1.1 the field of constants of $K\langle y_0,\hat{y}_1\rangle$ is $C$. Otherwise $\partial (\hat{y}_1 )=\partial( N_0 )$, $N_0\in K\langle y_0\rangle$. Derivating $y_1/y_0=\hat{y}_1$ we have $\partial (y_1 )/y_1=\partial (y_0 )/y_0 +\partial (  \hat{y}_1 )$, and hence $\partial (y_1 )/y_1\in K\langle y_0\rangle$. 

Moreover if the hyperexponential $y_1$ is a logarithmic derivative of a $ K\langle y_0\rangle$-radical then, by \eqref{eq-Risch}, $y_1\in K\langle y_0\rangle$, which is a contradiction. Thus {Theorem \ref{thm-Bron} (\cite{Bron}, Theorem 5.1.2)} implies that the field of constants of $K\langle y_0, y_1\rangle$ is the field of constants of $ K\langle y_0\rangle$, that is $C$.
\end{proof}

\para

\begin{rem}{
We would like to point out that the classical solutions of the Lam\'e equation co\-rres\-pond exactly to the case contemplated here when $ \mu = 0 $, see \cite{WW} and the references therein. Then their Picard-Vessiot extensions are given by { \eqref{thm-especializado_0}}. Their Galois groups are studied in \cite{Morales}}
\end{rem}

\para

For all but a finite number of points in $\Gamma_s$, we summarize the situation in  the next commutative diagram:

\centerline{
\xymatrixcolsep{5pc}{
\xymatrix{
C(\Gamma_{s}) \,\, \ar@{^-->}[r] &K(\Gamma_s)\,\, \ar@{^-->}[r]^-{\mbox{Liouvillian}} & \,\, K(\Gamma_s)\langle \Psi_s\rangle  \\
C \,\, \ar@{^-->}[u] \ar@{^-->}[r] &K\,\,  \ar@{^-->}[u] \ar@{^-->}[r]^-{\mbox{at } P_{0}\in \Gamma_{s}} & \,\, K\langle y_0\rangle  \ar@{^-->}[u]}
}}
\noindent for $K=C\langle u_{s}\rangle$. {In fact, in previous notations, this diagram holds for each $P_0 \not\in Z_s$ in \eqref{eq-Zs}.} In other words, for the non-branch points of $ \Gamma_s $, we have a good specialization process to obtain the classical Picard-Vessiot field from the spectral Picard-Vessiot  field.

\section{One-parameter spectral PV fields for rational curves}\label{sec-Parametric}

 In this section we assume that the spectral curve $\Gamma_s$ of the pair $\{L_s,A_{2s+1}\}$ is an  affine rational curve, i.e. its closure in $\mathbb{P}^2$ is birrationally equivalent to $\mathbb{P}^1$ (\cite{Harr}, pp. 78). Hence an open subset {$\tilde{U}$ of $\Gamma_s$ is isomorphic to a Zariski open subset $U$ of the affine line $\mathbb{A}^1$}. Then  we have
\begin{equation}
    \mathbb{C}(\Gamma_s ) \overset{\rho_1}{\simeq} \mathbb{C}(\tau ) \quad \mbox{ with }\rho_1(\lambda)=\chi_1(\tau) \mbox{ and }\rho_1(\mu)=\chi_2(\tau)
\end{equation}
for a complex parameter $\tau$ such that
\begin{equation}
\aleph: U\ni \tau \rightarrow    \left(\chi_1(\tau),\chi_2(\tau)\right)\in \Gamma_s
\end{equation}
is a regular isomorphism, and then $\mathbb{C}(\Gamma_s )  \simeq   \mathbb{C}(U )$.

\para

Let us consider next the differential structure of  $K(\tau)=\mathbb{C}\langle u_s \rangle(\tau)$.
Recall that $\tau$ is an alge\-braic indeterminate over $ \mathbb{C}\langle u_s \rangle$, which allows to extend the derivation $\partial=d/dx$ of $\mathbb{C}\langle u_s \rangle$ to   $\mathbb{C}\langle u_s \rangle (\tau)$ since $\partial(\tau)=0$. Then we extend the isomorphism $\rho_1$ to an isomorphism  of differential fields
$
    \mathbb{C}(\Gamma_s )\langle u_s \rangle  \simeq \mathbb{C}(\tau )\langle u_s \rangle.
$
Since $\mathbb{C}(\tau )\langle u_s \rangle=\mathbb{C}\langle u_s \rangle  (\tau )$, the composition of $\rho_1$ with the natural isomorphism
$
    \mathbb{C}\langle u_s \rangle (\Gamma_s )  \simeq  \mathbb{C}(\Gamma_s )\langle u_s \rangle
$
gives an isomorphism, let us call it also $\rho_1$,
\begin{equation}
    K(\Gamma_s)=\mathbb{C}\langle u_s \rangle (\Gamma_s )  \simeq   \mathbb{C}\langle u_s \rangle  (\tau )=K(\tau).
\end{equation}
The next commutative diagram of differential fields illustrates the differential algebraic situation:
\begin{equation}\label{diag_rho}
\xymatrix{
K(\Gamma_s)\,\, \ar[r]^{\rho_1} & \,\, K(\tau) \\
\bbC(\Gamma_s)\,\,  \ar@{^(->}[u] \ar[r]^(.5){\rho_1} & \,\, \bbC(\tau) \ar@{^(->}[u]
}
\end{equation}
Moreover, $\rho_1$ extends naturally to an isomorphism $\varrho$ between the rings of differential operators $K(\Gamma_s)[\partial]$ and $K(\tau)[\partial]$.

\para

Using Proposition \ref{thm-L1}, we obtain the right common factor ${\varrho}(\partial-\phi_s)$ of
$L_s-\chi_1(\tau)={\varrho}(L_s-\lambda)$ and $A_{2s+1}-\chi_2(\tau)={\varrho}(A_{2s+1}-\mu)$ in $K(\tau)[\partial]$. Let us define $\tilde{\phi}_s=\rho_1(\phi_s)$. Observe that $\tilde{\phi}_s$ is a nonzero element of $K(\tau)$ since, by Proposition \ref{thm-L1}, $\phi_s$ is nonzero in $K(\Gamma_s)$.
Furthermore, since the isomorphism respects the ring structure, from \eqref{eq-factKGamma} we obtain
\begin{equation}\label{eq-factKtau}
    L_s-\chi_1(\tau)=(-\partial-\tilde{\phi}_s)(\partial-\tilde{\phi}_s),\mbox{ in } K(\tau)[\partial].
\end{equation}
Furthermore, we have factorizations
\[L_s-\chi_1(\tau)=(-\partial-\tilde{\phi}_+)(\partial-\tilde{\phi}_+)=(-\partial-\tilde{\phi}_-)(\partial-\tilde{\phi}_-)\]
where $\tilde{\phi}_+=\tilde{\phi}_s$ and $\tilde{\phi}_-=\rho(\phi_-)$ are distinct solutions of the same  Riccati equation $\partial(\phi)+\phi^2=u_s-\chi_1(\tau)$, since $\rho_1$ respects the differential field structure. In the differential closure $\widehat{K(\tau)}$ of $K(\tau)$ \cite{Kolchin}, we consider nonzero solutions $\Upsilon_+$ and $\Upsilon_-$ respectively of the ordinary differential equations
\begin{equation}\label{eq-Upsilon}
\partial(\Upsilon)=\tilde{\phi}_+\Upsilon\,\,\,\mbox{ and }\,\,\,\partial(\Upsilon)=\tilde{\phi}_-\Upsilon.
\end{equation}

\begin{lem}\label{lem-solpar}
Let $\Upsilon_+$ and $\Upsilon_-$ as in \eqref{eq-Upsilon}, it holds that:
\begin{enumerate}
\item $\{\Upsilon_+,\Upsilon_-\}$ is a fundamental set of solutions of $(L_s-\chi_1(\tau))(\Upsilon)=0$.
\item $\Upsilon_+\Upsilon_-\in K(\tau)$.
\end{enumerate}
\end{lem}
\begin{proof}
The proof is analogous to the proof of Lemma \ref{lem-sol}, noting that
\[\frac{w(\Upsilon_+,\Upsilon_-)}{\Upsilon_+\Upsilon_-}=\tilde{\phi}_+-\tilde{\phi}_-=\rho_1 (\phi_+-\phi_-)=\rho_1 \left(\frac{ 2}{\varphi}\mu\right)\neq 0\]
since $\frac{ 2}{\varphi}\mu\neq 0$ and $\rho_1$ is an isomorphism.
\end{proof}

By Lemma \ref{lem-solpar}, 2 we have $K(\tau)\langle \Upsilon_+,\Upsilon_-\rangle=K(\tau)\langle \Upsilon_+\rangle=K(\tau)\langle \Upsilon_-\rangle$.
We denote $\Upsilon_s=\Upsilon_+$.
Now we can apply the differential algebraic results in \cite{Bron}, to the hyperexponential $\Upsilon_s$ and the differential field $(K(\tau),\partial)$, regarding the integration problem of $\tilde{\phi}_s=\partial \Upsilon_s/\Upsilon_s$ in $K(\tau)$.

\begin{lem}\label{lema-Krad-param}
The element $\tilde{\phi}_s$ of $K(\tau)$ is not a logarithmic derivative of a $K(\tau)$-radical.
\end{lem}

\begin{proof}
Let us assume that there exists $w\in K(\tau)$, $w\neq 0$, $\partial(w)\neq 0$ such that $\tilde{\phi}_s=\frac{\partial (w)}{nw}$ for a nonzero integer $n$. Since $\rho_1$ is an isomorphism, $w=\rho_1(v)$, with $v\in K(\Gamma_s)$, $v\neq 0$, $\tilde{\partial}(v)\neq 0$. Then
$$
\tilde{\phi}_s=\rho_1(\phi_s)=\rho_1\left(\frac{\tilde{\partial} (v)}{nv}\right)=\frac{\partial (w)}{nw}
$$
implies $\phi_s=\frac{\tilde{\partial} (v)}{nv}$ contradicting Lemma  \ref{prop-kradical}.
\end{proof}

\begin{thm}\label{thm-Bron2}
Let $L_s$ be an algebro-geometric Schr\" odinger operator with rational spectral curve $\Gamma_s$ parametrized by $(\chi_1 (\tau),\chi_2 (\tau))\in \bbC(\tau)^2$. Let $\partial-\tilde{\phi}_s$ be the intrinsic right factor of $L_s-\chi_1(\tau)$ as in \eqref{eq-factKtau}.
The Liouvillian extension $K(\tau)\langle \Upsilon_s\rangle$ of $K(\tau)$, by a nonzero solution $\Upsilon_{s} \in \widehat{K(\tau)}$ of $(\partial-\tilde{\phi}_s )\Upsilon =0$, is a transcendental extension with field of constants $\bbC(\tau)$.
\end{thm}

\begin{proof}
By Lemma \ref{lema-Krad-param} and  \cite{Bron}, Theorem 5.1.2 the result follows.
\end{proof}

{In Theorem {\bf C} of the introduction}, 
we gave the field structure of the spectral Picard-Vessiot field for any spectral curve. {In Theorem {\bf D} of the introduction,}  we give its parametric presentation whenever the spectral curve is a rational curve. {We are ready to prove Theorem {\bf D}.}

\para

{
{\bf Proof of Theorem D.}}
{\it
The isomorphism $\rho_1$ in \eqref{diag_rho} extends naturally to an isomorphism between differential fields

$$
\hat{\rho}_1 : K(\Gamma_s)\langle \Psi_s\rangle\longrightarrow K(\tau)\langle \Upsilon_s \rangle
$$
by sending a nonzero solution $\Psi_s$ of $(\partial-\phi_s )\Psi=0$, as in Theorem \ref{thm-Bron}, to $\Upsilon_s $. Hence diagram \eqref{diag_rho}   extends to the following commutative diagram of differential fields, whose second row shows the fields of constants:
$$
\xymatrix{
K(\Gamma_s)\langle \Psi_{s}\rangle\,\, \ar[r]^{\hat{\rho}_{1}} & \,\, K(\tau)\langle\Upsilon_{s}\rangle \\
\bbC(\Gamma_{s})\,\,  \ar@{^(->}[u] \ar[r]^{\rho_{1}} & \,\, \bbC (\tau) \ar@{^(->}[u]
}
$$
Hence we have proved the  required statement.
}


\para

The global rational parametrization of the spectral curve allows us to be more {specific} about the integral representation of $\Upsilon_s$, which was defined in \eqref{eq-Upsilon} as a solution of $(\partial-\bar{\phi}_s)\Upsilon=0$.  {Moreover, we have the following sequence of differential fields, 
\begin{equation}\label{eq-tower}
   \bbC(\Gamma_s ) = \bbC(\tau ) \subset {\bbC \langle u_s \rangle (\tau )=K(\tau) \subset K(\tau) \langle\Upsilon_s \rangle} .
\end{equation}
with $\Upsilon_s$ a transcendental element over $\bbC \langle u_s ,\tau \rangle$, by Theorem \ref{thm-Bron2}.} We can finally show the advantages of constructing solutions of the spectral problem \eqref{eq-problem} using a global rational parametrization of the spectral curve in $\bbC(\tau)^2$, instead of a local parametrization by Puiseux series in  $\bbC\langle\langle \tau \rangle\rangle^2$. We do so by means of a family of Rosen-Morse potentials in Example \ref{example} to illustrate Remark \ref{rem-symbolicint}, and establishing the appropriate algebraic setting to solve problem \eqref{eq-problem} analytically obtaining Theorem \ref{thm-analytic}. 

\begin{rem}\label{rem-symbolicint}
Recall that $\partial-\tilde{\phi}_s$ is an operator in $K(\tau)[\partial]$ where
\[K(\tau)=\bbC\langle u_s\rangle (\tau)=\bbC(\tau)\langle u_s\rangle.\]
If $u_s$ is a monomial over the differential field $\bbC(\tau)$ one can address the integration of $(\partial-\bar{\phi}_s)\Upsilon=0$ in the differential algebraic setting of \cite{Bron}, Chapter 5.
Then the differential integration theorems and algorithms in \cite{Bron} can be used to compute $\Upsilon_s$ in an elementary extension of $\bbC(\tau)$ if it exists.
\end{rem}

\begin{ex}\label{example}
Let us consider the family of Rosen-Morse potentials \begin{equation}\label{eq-RM}
  u_s=\frac{-s(s+1)}{\cosh^2(x)},\,\,\, s\geq 1,
\end{equation}
which belong to the differential field $K=\bbC\langle \cosh(x)\rangle=\bbC(e^x)$, with derivation $\partial=d/dx$ and field of constants $\bbC$. It is well known that the Schr\" odinger operators $-\partial^2+u_s$ are algebro-geometric, \cite{Veselov}, \cite{MRZ}.
In \cite{MRZ}, we gave algorithms to compute $A_{2s+1}$. By means of the differential resultant $\dres(L_s-\lambda, A_{2s+1}-\mu)$ the defining polynomial $f_s$ of $\Gamma_s$ can be computed,
{$$f_s (\lambda ,\mu ) =\mu^2+\lambda\prod_{\kappa=1}^s(\lambda+\kappa^2)^2,$$}
see for instance \cite{GH}, Example 1.31.
By means of the first differential subresultant of $L_s-\lambda$ and  $A_{2s+1}-\mu$ their right common factor $\partial-\phi_s$ is obtained. For $s=1$ 
\[
\phi_1=\frac{\mu+\frac{1}{2}\partial(\varphi)}{\varphi}=\frac{(z^2 +1)^3\mu+z^4-z^2}{(z^2+1) ((z^2+1)^2\lambda+z^4+z^2+1)}, \mbox{ where } \varphi=\lambda+1-\frac{1}{\eta^2},
\]
with $\eta=\cosh(x)$ and $z=e^x$.
All the curves $\Gamma_s$ are rational, in particular they admit a polynomial global parametrization
\begin{equation}\label{eq-param-RM}
\aleph_s(\tau)=(\chi_1(\tau),\chi_2(\tau))=\left(-\tau^2 ,-\tau\prod_{\kappa=1}^s(\tau^2-\kappa^2)
\right),
\end{equation}
{with $\tau$ transcendental over $\bbC (e^x )$. }Replacing $\lambda=\chi_1(\tau)$ and $\mu=\chi_2(\tau)$ in $\phi_s$ we obtain $\tilde{\phi}_s(x,\tau)$, which is a rational function in $\bbC(\tau)(z)$. Since $z=e^x$ is transcendental over $\bbC(\tau)$, we are in the situation of Remark \ref{rem-symbolicint}. Using the symbolic integration package of Maple 18 to obtain $\Upsilon_s$ as ${\tt int}(\tilde{\phi}_s,x)$. For instance the primitive of $\tilde{\phi}_1$ equals
\[
\Upsilon_1=\displaystyle{\frac{\left(\tau-1 \right) z^2+\tau+1}{z^2+1}}{e}^{x\tau} ,
\]
replacing $\lambda=-\tau^2$ and $\mu=-\tau (\tau^2 - 1)$ in $\phi_1$; observe that all functions are analytic outside the analytic set $E=\left\{ \   z^4+z^2+1 -(z^2+1)^2\tau^2  =0 \ \right\} \subset \bbC^2$. {By \eqref{eq-tower} the corresponding sequence of differential fields is:
 \begin{equation*}
   \bbC(\Gamma_1 ) = \bbC(\tau ) \subset \bbC\langle e^x \rangle (\tau)\subset \bbC (e^x ,\tau) \langle\Upsilon_1 \rangle ,
\end{equation*}
since $\bbC\left< \frac{-2}{\cosh^2(x)}  \right> (\tau) =\bbC(e^x ,\tau )$. Moreover, Theorem \ref{thm-Bron2} guaranties that the function $\Upsilon_1$ is trans\-cendental over $\bbC (e^x ,\tau) $, and Theorem {\bf D} that  $\bbC (e^x ,\tau) \langle\Upsilon_1 \rangle$ is the minimal field extension where the solutions of the spectral problem 
\begin{equation}
  \left(  -\partial^2 +\frac{-2}{\cosh^2(x)}\right) \Psi = -\tau^2 \Psi
\end{equation}
 can be expressed as an element of a differential field with field of constants  $\bbC (\tau )$. In fact, any solution of the above problem is of the form
 \[
 \Psi (x, \tau ) = c_1 (\tau ) \Upsilon_1  (x, \tau )+c_2 (\tau ) \Upsilon_{-} (x,\tau) 
 \]
 with $c_i (\tau ) \in \bbC (\tau )$ \  and \ $\Upsilon_{-} (x,\tau) = \frac{\varphi}{ 2\mu \Upsilon_1 }=\frac{ (-(z^2+1)^2 \tau^2 +z^4+z^2+1)(z^2 +1)^2}{-2\tau (\tau^2 -1 )(\left(\tau-1 \right) z^2+\tau+1)}e^{-x\tau }$.}
\end{ex}

\para

Next we will establish the appropriate algebraic setting to solve the spectral problem \eqref{eq-problem} {analytically}.

For an analytic potential $u_s(x)$ in a complex domain $D$ in $\mathbb{C}$, we will prove  that $(\partial-\bar{\phi}_s)\Upsilon=0$ has an analytic solution in an open neighborhood of each point in some open subset of $\bbC^2$.
Since $\Gamma_s$ is  defined by $\mu^2+R_{2s+1}(\lambda)$, its singular locus $\Sing (\Gamma_s)$ is contained in the finite set $Z_s = \{ (\lambda,0)\in \mathbb{C}^2 \ : R_{2s+1}(\lambda)=0 \}$, and we can assume that $U\cap Z_s =\emptyset$. As a consequence, at any fixed $\tau_0 \in U$ we have
\begin{equation}
 \chi_1 (\tau) =\dfrac{g_1 (\tau ) }{h_1 (\tau )}
\quad , \quad
\chi_2 (\tau) =\dfrac{g_2 (\tau ) }{h_2 (\tau )}
\end{equation}
for some polynomials $g_i , h_i $ in $\mathbb{C}[\tau ]$ and a neighborhood $U_{\tau_0 }$ of $\tau_0$ in $U$ where $h_i (\tau )\not=0$ .  Then, in $U$, we break down the spectral problem:
\begin{equation}\label{problem-Parametric}
\left(-\partial^2 +u_s(x)  -\chi_1(\tau)\right)\Upsilon =0 \ , \ \left(A_{2s+1} -\chi_2(\tau) \right)\Upsilon=0 \ 
\end{equation}
into two steps. First we compute the right common factor of the operator $-\partial^2 +u_s(x)  -\chi_1(\tau)$ and $A_{2s+1}-\chi_s(\tau)$ with $\tau$ in $U$ to obtain a one parameter family of common right first order factors, say $\partial- \tilde{\phi}_s (x , \tau )$. Secondly, we address the resolubility of the first-order equation that generates the one dimensional space of common solutions of  \eqref{problem-Parametric}:
\begin{equation}\label{eq-factorRacional}
\left(\partial- \tilde{\phi}_s (x , \tau )\right)\Upsilon =0   ,\,\, \mbox{ for } x\in D, \,\, \tau\in U .
\end{equation}
Finally we can give a fundamental matrix for the linear differential equation  $(-\partial^2 +u_s(x)  -\chi_1(\tau))\Upsilon =0$ that varies continuously in $\tau \in U$, see Theorem \ref{thm-analytic}.

\begin{thm}\label{thm-analytic}
Let $-\partial^2+u_s(x)$ be an algebro-geometric Schr\" odinger operator, with $u_s(x)$ an analytic potential in a complex open set $D$. Let us assume that the spectral curve $\Gamma_s$ of $L_s-\lambda$ is a rational curve. Then, for  each $c$ in a domain $\Omega_c \subset D\times U \subset \bbC^2$, the equation
 \begin{equation} \label{eq-BA-analytic}
     (\partial- \tilde{\phi}_s (x , \tau ))\Upsilon=0
 \end{equation}
 has an analytic solution $\Upsilon (x,\tau)$ in a subdomain of $\Omega_c$.
\end{thm}
\begin{proof}
Having a rational parametrization of the spectral curve $\Gamma_s$ allows to write 
\begin{equation}\label{eq-phianalytic}
    \tilde{\phi}_s (x , \tau )=\frac{N_1(x,\tau)}{N_2(x,\tau)}
\end{equation}
 with $N_1,N_2\in \bbC\{u_s(x)\}[\tau]$. For some open domain $D_1\subseteq D$ then \eqref{eq-phianalytic} is a well defined analytic function in $(D_1 \times U) \backslash E\subset \bbC^2$ where $E$ is the analytic set $\{N_2(x,\tau)=0\}$. 
 
 For a fixed $c=(x_0,\tau_0)\in (D_1 \times U) \backslash E$ there exists an open domain $\Omega_c\subset (D_1 \times U) \backslash E$ where the differential equation 
 \begin{equation}
     (\partial- \tilde{\phi}_s (x , \tau ))\Upsilon=0
 \end{equation}
 has an analytic solution $\Upsilon (x,\tau)$ in an open subdomain $\widetilde{\Omega}_c$ of $\Omega_c$ (see \cite{Hi}).
\end{proof}

\begin{rem}\label{rem-sheaf}
Observe that Theorem \ref{thm-analytic} guaranties the existence of a sheaf structure of spectral Picard-Vessiot fields, $\pmb{\Lambda}=\pmb{\Lambda} (\Gamma_s) $, on a tubular neighborhood $W\subset\mathbb{C}^2$ of $\Gamma_s$,  possibly outside of an analytic set ( $E$ in the proof of the previous theorem ). For almost all $c\in W$, each fiber $\pmb{\Lambda}_c$ is isomorphic to the differential field of germs at $c$  of meromorphic functions on the spectral curve $\Gamma_s $ over $ K$ extended with a solution $\Upsilon_c$ of \eqref{eq-BA-analytic}. Moreover, when $W'\subset W$ is quasi-compact, the space of sections $\Gamma (W',\pmb{\Lambda})$ has a field structure and then $\pmb{\Lambda}$ is a torsion free sheaf over $W$. The relation between sheaf $ \pmb{\Lambda}$, the $K$-points of \  $\Gamma_s$  \ ,  and its singular locus over  $\mathbb{C}$, it is an intriguing question that we will study in future works.
\end{rem}

\section{Concluding remarks}\label{sec-Conclusiones}

We have generalized the concept of (classical) Picard-Vessiot extension of $L_s-\lambda_0$ over $K$, with algebraically closed constant field $C$, to define a spectral Picard-Vessiot field for $L_s-\lambda$ over $K(\Gamma_s)$ with constant field $C(\Gamma_s)$, which is not necessarily algebraically closed. These spectral Picard-Vessiot fields admit a specialization process at each $P_0 =( \lambda_0 ,\mu_0 )$ in $\Gamma_s$ that allowed us to give the (classical) Picard-Vessiot extensions $\Sigma_{P_0 } / K$ of $L_s-\lambda_0$ in Section \ref{sec-specialization}. In addition, one could analyze the Differential Galois group ${Gal}(\Sigma_{P_0 } / K)$ with the approaches given in \cite{AMW}, \cite{Brez} and  \cite{UW}.

\para 

Whenever the spectral curve $\Gamma_s$ is hyperelliptic, a solution $\Psi_s$ of problem \eqref{eq-problem}, is expressed in \cite{GH}, Theorem 1.20 in terms of $\Theta$-functions, this is the well known Baker-Akheizer function. Also in the hyperelliptic case, the work of Brezhnev (\cite{Brez} and references there in) combines the expression of $\phi_s$ in terms of $\Theta$-functions, associated to the spectral curve, with the definition of classical Picard-Vessiot extensions. It remains open to combine the new spectral Picard-Vessiot structure we have defined   
with the extensively developed $\Theta$-function approach for the hyperelliptic case. 

\para

In Section \ref{sec-Parametric} we studied the case of rational spectral curves transforming the original spectral problem \eqref{eq-problem} into a spectral problem 
\begin{equation}
(-\partial^2+u_s-\chi_1(\tau))\Upsilon=0    
\end{equation}
with a free parameter $\tau$. In this case the spectral Picard-Vessiot field of \eqref{eq-problem} is isomorphic to a transcendental Liouvillian extension $K(\tau)(\Upsilon_s)$.  In this context much more can be said about the hyper\-exponential $\Upsilon_s$. In the coefficient field $K(\tau)=C(\tau)\langle u_s\rangle$, the differential integration theorems and algorithms of M. Bronstein in \cite{Bron}, and other references there in, can now be used to compute $\Upsilon_s$. In addition we showed how, in the case of analytic potentials, the algebraic techniques developed in Section \ref{sec-Picard-Vessiot} (for all spectral curves) combined with a global rational parametrization of the curve, allowed us to solve the spectral problem \eqref{eq-problem} analytically in closed form when the spectral curve is a rational curve and has singular points, see \ref{thm-analytic} and \ref{rem-sheaf}.

\vskip1cm

\noindent {\sf Acknowledgments:} We kindly thank all members of the Integrability Madrid Seminar for many fruitful discussions: J. Capit\'an, R. Hern\'andez Heredero, S. Jim\'enez, A. P\' erez-Raposo, J. Rojo Montijano and R. S\'anchez; the member in Colombia: D. Bl\' azquez-Sanz, and the member in Rep\'ublica Dominicana: P.B. Acosta-Hum\' anez. In particular:  to A. P\' erez-Raposo for carefully proof reading this manuscript. We also thank J. P. Ramis and E. Paul, from University of Toulouse 3 - Paul Sabatier, for stimulating discussion on this kind of problems. {  The authors would like to thank the anonymous referee who helped to improve the final version of this work.}

\para

The first two authors are members of the Research Group ``Modelos ma\-tem\'aticos no li\-neales", UPM and S.L. Rueda has been
partially supported by the ``Ministerio de Econom\'\i a y Competitividad"  under the project MTM2014-54141-P. M.A. Zurro is partially supported by Grupo UCM 910444.

\appendix

\vskip1cm

\noindent{\bf Appendix}

\para

\begin{appendix}

\renewcommand{\thesection}{\Alph{section}}%
\section{KdV potentials}\label{sec-KdV}

The purpose of this appendix is to define KdV potentials and prove that $-\partial^2+u_s$ is algebro-geometric only for a KdV potential $u_s$.
We present the KdV hierarchy and the family of diffe\-rential operators of its Lax representation with the language of differential algebra \cite{Ritt}. The KdV-hierarchy was studied for the first time in the paper \cite{GD}. We follow the normalization in \cite{GH}, see also \cite{MRZ}, Section 3.

\para

Let us consider a differential indeterminate $u$ over $C$. We will call {\it formal Schr\" odinger ope\-rator} to the operator $L(u)=-\partial^2+u$ with coefficients in the ring of differential polynomials $C\{u\}=C[u,u',u'',\ldots]$, where $u'$ stands for $\partial(u)$ and $u^{(n)}=\partial^n(u)$, $n\in\bbN$, where $\bbN$ is the set of positive integers including $0$.

\para

Let us consider the pseudo differential operator
\begin{equation}\label{eq-recursion}
\cR=-\frac{1}{4}\partial^2+u+\frac{1}{2}u'\partial^{-1}\mbox{ and its {(formal)} adjoint }
\cR^*=-\frac{1}{4}\partial^2+u-\frac{1}{2}\partial^{-1}u',
\end{equation}
in the ring of pseudo-differential operators in $\partial$ with coefficients in $C\{u\}$  (see \cite{Good})
where $\partial^{-1}$ is the inverse of $\partial$, $\partial^{-1}\partial=\partial\partial^{-1}=1$. {Recall that the (formal) adjoint operator of $T= \sum_{k=-d}^{m} a_k \partial^k$, $a_k\in C\{u\}$ is defined by $T^*=  \sum_{k=-d}^{m} (-1)^{k} \partial^k a_k$, see \cite{Olver}, Theorem 5.31.}

Observe that $\cR^*=\partial^{-1}\cR\partial$.
The operator $\cR^*$ is a recursion operator of the KdV equation (see \cite{Olver}, p. 319).
Applying the recursion operator $\cR$, we define:
\begin{equation}\label{eq-kdv}
\kdv_0:=u',\,\,\, \kdv_n:=\cR(\kdv_{n-1}),\mbox{ for }n\geq 1.
\end{equation}
Applying $\cR^*$ we define:
\begin{equation}\label{eq-fn}
v_0:=1,\,\,\, v_n:=\cR^*(v_{n-1}),\mbox{ for }n\geq 1.
\end{equation}
Hence for $n\in\bbN$ it holds
$2\partial(v_{n+1})=\kdv_n$.
We will call the differential polynomials $\kdv_n$ the {\it KdV differential polynomials}. By \cite{MRZ}, Lemma 3.1, the formulas for  $\kdv_n$ and $v_n$ give differential polynomials in $C\{u\}$.

\para

As in \cite{GH}, we define a family of differential operators in $C\{u\}[\partial]$ of odd order (see also \cite{Dikii}, \cite{Novikov})
\begin{equation}\label{eq-A2s+1}
P_1(u):=\partial,\,\,\,P_{2n+1}(u):=v_{n}\partial-\frac{1}{2}\partial(v_{n})+P_{2n-1}(u)L(u),\mbox{ for }n\geq 1.
\end{equation}
The operators $P_{2n+1}(u)$ have the important property (see for instance \cite{MRZ}, Lemma 3.2)
\begin{equation}\label{eq-kdvcomm}
    [P_{2n+1}(u),L(u)]=\kdv_n(u)
\end{equation}
it is the multiplication operator by the $\kdv_n$ differential polynomial. This is the famous Lax representation of $\kdv_n$, see \cite{GH}, \cite{Novikov}.
We will call the differential operators $P_{2n+1}(u)$ the {\it KdV differential operators}.

\para

Note that after replacing the differential variable $u$ by a potential $u_s\in\Sigma$ we obtain a Schr\" o\-dinger operator $L_s=L(u_s)$ and differential operators  $\{P_{2n+1}^s=P_{2n+1}(u_s)\}_{n\geq 0}$ whose coefficients belong to the differential field
$K=C\langle u_s\rangle$ with derivation $\partial$ and field of constants $C$. The next theorem  tells us how to construct the partner $A_{2s+1}$ of an algebro-geometric $L_s$.

\begin{thm}\label{thm-partner}
Let $L_s=-\partial^2+u_s$ be an algebro-geometric Schr\" odinger operator of level $s$ with partner $A_{2s+1}$.
Then there exists a vector of constants ${\bf c}^s=(c_1^s,\ldots ,c_s^s)\in C^s$ such that
\[A_{2s+1}=P_{2s+1}(u_s)+c_1^s P_{2s-1}(u_s)+\cdots +c_{s}^s P_1(u_s).\]
Furthermore $u_s$ verifies the equation $\KdV_s(u,{\bf c}^s)=0$ of the KdV-hierarchy defined by the diffe\-rential polynomial in $C\{u\}$
\begin{equation}\label{eq-KdVs}
    \KdV_s(u,c^s)=\kdv_s(u)+c_1^s\kdv_{s-1}(u)+\cdots +c_s^s \kdv_0(u).
\end{equation}
\end{thm}
\begin{proof}
Let us assume that $A_{2s+1}$ is monic. Since $\{P_{2i+1}^s\}_{i\leq s}$ and $\{L_s^i\}_{i\leq s}$ are families of operators in $K [\partial]$ of odd and even orders less than $2s+1$ respectively, we divide $A_{2s+1}$ by those families and write
\begin{equation}\label{eq-decompA}
    A_{2s+1}=\sum_{i=0}^s q_{2i+1} P_{2i+1}^s+\sum_{i=0}^s q_{2i} L_s^i
\end{equation}
with $q_{2s+1}=1$ and $q_{2i+1},q_{2i}\in K$.
Let us compute next the commutator  of the right hand side of \eqref{eq-decompA} with $L_s$. Observe that $[a,L_s]=\partial^2 (a)+2\partial (a)\partial$, for $a\in K$ and
\[[q_{2i+1} P_{2i+1}^s,L_s]=-\partial^2(q_{2i+1})P_{2i+1}^s-2\partial(q_{2i+1})\partial P_{2i+1}^s+q_{2i+1}\kdv_i(u_s)\]
and
\[[q_{2i}L_s^i,L_s]=[q_{2i},L_s]L_s^i=(\partial^2(q_{2i})+2\partial(q_{2i})\partial)L_s^i.\]
 The only term of order $2i+2$ is the leading term of $\partial P_{2i+1}^s$ and the only term of order $2i+1$ is the leading term of $\partial L_s^i$.
Since $[A_{2s+1},L_s]=0$ then $\partial(q_{2i})=0$ and $\partial(q_{2i+1})=0$. Therefore $[q_{2i}L_s^i,L_s]=0$ and $q_{2i+1}=c_i^s\in C$, $i=0,\ldots ,n$ implies that
\[0=\sum_{i=0}^s c_i^s \kdv_i(u_s)\mbox{ and }A_{2s+1}=\sum_{i=0}^s c_i^s P_i^s,\]
which proves the result.
\end{proof}

\para

Let us review the converse of the previous theorem, which is a consequence of some well known result, see \cite{GH}:
For a vector ${\bf c}_n=(c_1,\ldots ,c_n)$ of algebraic indeterminates  it holds by \eqref{eq-kdvcomm} and \eqref{eq-KdVs} that,
\begin{equation}
    \left[
    L(u) \ , \ \sum_{i=0}^n c_i P_i(u)
    \right]
    =\KdV_n(u,{\bf c}_n),\,\,\mbox{ in } C\{u\}[\partial].
\end{equation}
Let as assume we are given $L_s=-\partial^2+u_s$ such that $u_s$ is a solution of an equation of the KdV-hierarchy \eqref{eq-KdVs}, which means the following.
There exists a vector of constant ${\bf c}^n=(c_1^n,\ldots ,c_s^n)\in C^n$ such that
after replacing $u$ by $u_s$ and  ${\bf c}_n$ by ${\bf c}^n$ then
\begin{equation}
    [L_s,\sum_{i=0}^n c_i^n P_i(u_s)]=\KdV_n(u_s,{\bf c}^n)=0 \,\,\mbox{ in } K[\partial].
\end{equation}
If $s$ is minimal with the previous property, then there exists a unique vector of constants ${\bf c}^s$ such that  $\KdV_s(u_,{\bf c}^s)=0$ and
,by \cite{MRZ}, Proposition 4.2,  $\cC(L_s)=C[L_s,A_{2s+1}]$. See for instance \cite{MRZ}, Section 4 for more details on the behaviour of the integration constants of the KdV-hierarchy.

\begin{defi}
We say that $u_s$ in $\Sigma$ is a {\it KdV-potential of KdV level $s$} if it verifies the $\KdV_s$ equation \eqref{eq-KdVs} of the KdV-hierarchy for a vector of constants ${\bf c}^s\in C^s$, that is $\KdV_s(u_s,{\bf c}^s)=0$, and $s$ is minimal with this property.
\end{defi}

\section{Differential Subresultant Theorem}\label{subsec-subres}

We summarize next the definition and some important properties of differential resultants and subresultants to be used in this article. We use mainly the presentation given in \cite{Ch}, see also  \cite{Li} and the recent report \cite{McW}.

\para

Let us consider differential operators $P$ and $Q$ of orders $n$ and $m$ respectively with coefficients in a differential integral domain $(\bbD,\partial)$, whose quotient field $\bbK$ is equipped with the same derivation $\partial$. The ring $\bbK[\partial]$ is a left Euclidean integral domain and therefore every left ideal is left principal. If $\ord(P)\geq \ord(Q)$ then $P=qQ+r$ with $\ord(r)<\ord(Q)$, $q,r\in \bbK[\partial]$. Let us denote by $\gcd(P,Q)$ the greatest common (left) divisor of $P$ and $Q$.

As in the commutative case of polynomials in one variable, $G=gcd(P,Q)$ can be expressed in a unique way as $G=AP+BQ$, where the orders of $A$ and $B$ satisfy the natural restrictions. The search for $A$ and $B$ is equivalent to the resolution of a linear system defined by a linear map

\begin{align*}
    S_k:&\hspace{1.4cm}\bbK^{n+m-2k}\hspace{1.4cm}
    &\longrightarrow \hspace{2.5cm} \bbK^{n+m-k}\\
    &\footnotesize{(a_{n-k-1},\ldots ,a_0 ,b_{m-k-1},\ldots ,b_0 )} &\mapsto \mbox{coefficients of } A P+B Q \ ,
\end{align*}

\noindent when searching for a $\gcd$ $G_k$ of order $k$, $k=0,1,\ldots ,N:=\min\{n,m\}-1$. The matrix $S_k(P,Q)$ of this linear map is the
coefficient matrix of the extended system of differential operator
\[\Xi_k=\{\partial^{m-1-k} P,\ldots \partial P, P, \partial^{n-1-k}Q,\ldots ,\partial Q, Q\}.\]
Observe that $S_k(P,Q)$ is a matrix with $n+m-2k$ rows and $n+m-k$ columns, with entries in $\bbD$. Let $G_k$ be the {\it determinant polynomial} of $S_k(P,Q)$ in $\bbD[\partial]$ (see \cite{Li}  or the proof of Theorem \ref{thm-sresIdeal} for the construction).

\begin{defi}
The {\it subresultant sequence} of $P$ and $Q$ is the sequence $\{G_k\}_{k=0}^N$ of differential operators in $\bbD[\partial]$. The {\it differential resultant of $P$ and $Q$} is the zero order operator
\begin{equation}
    \dres(P,Q):=G_0=\det(S_0(P,Q)).
\end{equation}
\end{defi}

\begin{thm}\label{thm-sresIdeal} (\cite{McW}, Theorem 3.10)
Let $P,Q\in \bbD[\partial]$. The resultant $\dres(P,Q)$ of $P$ and $Q$ belongs to the ideal $(P,Q)$ generated by $P$ and $Q$ in $\bbD[\partial]$.
\end{thm}

\para

The next thm gives us a method to look for $\gcd(P,Q)$ and is essential for the main results of this paper.

\begin{thm}[\cite{Ch}, Theorem 4. Differential Subresultant Theorem]\label{thm-subres}
Given differential operators $P$ and $Q$ in $\bbD[\partial]$,  $\gcd(P,Q)$ is a differential operator of order $r$ if and only if:
\begin{enumerate}
\item $G_k$ is the zero operator for $k=0, 1, ,\ldots ,r-1$ and,
\item $G_r$ is nonzero.
\end{enumerate}
Then $\gcd(P,Q)=G_r$.
\end{thm}

\begin{rem}\label{rem-gcd}
\begin{enumerate}
\item The $\gcd(P,Q)$ is nontrivial (it is not in $\bbK$) if and only if\break  $G_0=\dres(P,Q)=0$.

\item Given $G_r=\gcd(P,Q)$ then $P=\bar{P} G_r$ and $Q=\bar{Q} G_r$, $\bar{P},\bar{Q}\in \bbK[\partial]$.

\end{enumerate}
\end{rem}

In this paper, we will use the case of  differential operators $P=\partial^2+a_0$ and $Q=\partial^{2s+1}+\cdots +b_1\partial+b_0$, $s\geq 1$ in $K[\partial]$, for a differential field $(K,\partial)$. Observe that $P-\lambda$ and $Q-\mu$ are differential operators with coefficients in the differential domain $\bbD=K[\lambda,\mu]$. The differential resultant
\begin{equation}\label{eq-G0}
  G_0=\dres(P-\lambda,Q-\mu)=-\mu^2-\lambda^{2s+1}+\cdots
\end{equation}
is the determinant of the matrix $S_0(P-\lambda,Q-\mu)$, whose coefficients are in $K[\lambda,\mu]$ and it belongs to the elimination ideal
$$(P-\lambda,Q-\mu)\cap K[\lambda,\mu].$$

The first subresultant $G_1$ is a differential operator of order $1$, the determinant polynomial of the matrix  $S_1(P-\lambda,Q-\mu)$, whose rows are the coefficients of
$$\Xi_{1}=\{\partial^{2s} (P-\lambda), \partial(P-\lambda), P-\lambda, Q-\mu\}$$
and whose columns are indexed by $\partial^{2s+1} ,\cdots ,\partial, 1$. The first subresultant equals
\begin{equation}\label{eq-G1}
    G_1=\varphi_1+\varphi_2\partial,
\end{equation}
with $\varphi_1=\det(S_1^0)$ and $\varphi_2=\det(S_1^1)$,  where $S_1^0$ and  $S_1^1$ are the submatrices of $S_1$ obtained by removing columns indexed by $\partial$ and $1$ respectively. Observe that
\begin{equation}\label{eq-S10S11}
    \det(S_1^0)=-\mu-\alpha(\lambda)\mbox{ and }\det(S_1^1)=\varphi_2(\lambda),
\end{equation}
for polynmials $\alpha,\varphi_2\in K[\lambda]$.

\begin{ex}\label{ex-subresLP3}
Given $P=\partial^2+a_0$ and $Q=\partial^{3}+b_2\partial^2+b_1\partial+b_0$ in $K[\partial]$. The differential resultant $\dres(P-\lambda,Q-\mu)$ is the determinant of the matrix $S_0(P-\lambda,Q-\mu)$,
whose rows are the coefficients of the polynomials in $\Xi_0=\{\partial^2 (P-\lambda),\partial (P-\lambda), P-\lambda, \partial (Q-\mu), Q-\mu\}$ and whose columns are indexed by $\partial^4,\ldots ,\partial,1$.
Take now that coefficient matrix
\[S_1(P-\lambda,Q-\mu)=\left(
\begin{array}{cccc}
1 & 0 & a_0-\lambda & a_0'\\
0 & 1 & 0 & a_0-\lambda\\
1 & b_2 & b_1 & b_0-\mu
\end{array}
\right)
\]
of $\Xi_1=\{\partial P, P-\lambda, Q-\mu\}$. The first subresultant $G_1=\det(S_1^0)+\det(S_1^1)\partial$ where
\[S_1^0=\left(
\begin{array}{ccc}
1 & 0 & a_0'\\
0 & 1 & a_0-\lambda\\
1 & b_2  & b_0-\mu
\end{array}
\right),\,\,\,
S_1^1=\left(
\begin{array}{ccc}
1 & 0 & a_0-\lambda \\
0 & 1 & 0 \\
1 & b_2 & b_1
\end{array}
\right).\]
\end{ex}

\nocite{*}

\end{appendix}

\bibliographystyle{cdraifplain}


\begin{thebibliography}{00}

\bibitem{AMW} Acosta-Hum\' anez, P.B., Morales-Ruiz, J.J., Weil, J.A., 2011. Galoisian approach to integrability of Schr\" odinger equation. Reports on Mathematical Physics 67 (3), 305-374.

\bibitem{Arr} Arreche, C., 2016. On the computation of the parameterized differential Galois group for a second-order linear differential equation with differential parameters. Journal of Symbolic Computation, Volume 75, 25-55.

\bibitem{BeardonNg} Beardon, A. F., Ng, T. W., 2006. Parametrizations of algebraic curves. In
Annales-Academiae Scientiarum Fennicae Mathematica, Vol. 31, No. 2, p. 541. Academia Scientiarum Fennica.

\bibitem{BB} Belokolos,E.D.,  Bobenko,A.I., Enolski, V.Z.,  Its, A.R.,  Matveev,V.B., 1994.  Algebro-geometric Approach in the Theory of Integrable Equations, Springer Series in Nonlinear Dynamics, Springer, Berlin.

\bibitem{BEG} { Braveman, A., Etingof, P., Gaitsgory, D., 1997. Quantum integrable systems and differential Galois theory. Transformation Groups 2, 31-56.}

\bibitem{Brez2} Brezhnev, Y.V., 2008. On the uniformization of algebraic curves. Moscow Matematical Journal, vol. 8, n. 2, 233-271.

\bibitem{Brez3} Brezhnev, Y.V., 2012. Spectral/quadrature duality: Picard-Vessiot theory and finite-gap potentials. Contemporary Mathematics, 563, 1.

\bibitem{Brez} Brezhnev, Y.V., 2013. Elliptic solitons, Fuchsian equations, and algorithms. St. Petersburg Mathematical Journal, 24(4), 555-574.

\bibitem{Bron} Bronstein, M. (2013). Symbolic integration I: transcendental functions (Vol. 1). Springer Science \& Business Media.

\bibitem{BC} Burchnall, J.L., Chaundy, T.W., 1928. Commutative ordinary differential operators. Proc. R. Soc. A 118, 557-583.

\bibitem{BC2} Burchnall, J.L., Chaundy, T.W., 1931. Commutative ordinary differential operators II. The Identity Pn = Qm. Proc. R. Soc. A 134, 471-485.

\bibitem{CS} Cassidy, P.J. and Singer, M.F., 2006. Galois theory of parameterized differential equations and linear differential algebraic groups. Differential Equations and Quantum Groups (IRMA Lectures in Mathematics and Theoretical Physics Vol. 9), ed. D. Bertrand, B. Enriquez, C. Mitschi, C. Sabbah, R. Schaefke, EMS Publishing house  pp. 113- 157

\bibitem{Ch} Chardin, M., 1991.
Differential Resultants and Subresultants.
Proc. FCT'91, Lecture Notes in Computer Science, 529, Springer-Verlag.

\bibitem{Cox} Cox D., Little J., O'Shea D., (1997). Ideals, Varieties,
and Algorithms (2nd ed.). Springer-Verlag,  New York.

\bibitem{CH} Crespo, T., Hajto, Z., (2011). Algebraic Groups and Differential Galois Theory, Amer. Math. Soc., Providence, Rhode Island.











\bibitem{Drach2} Drach, J., 1919.  D\'etermination des cas de r\'eduction de l'\'equation diff\'erentielle $d^2y/dx^2= [\phi(x) + h] y$.  C. R. Acad. Sci.
Paris 168, 47-50.

\bibitem{Drach3} Drach, J., 1919.  Sur l'intÃƒÂ©gration par quadrature de l'ÃƒÂ©quation $d^2y/dx^2= [\phi(x) + h] y$.  C. R. Acad. Sci.
Paris 168,  337-340.

\bibitem{Dikii} Dickey, L. A. (2003). Soliton equations and Hamiltonian systems (Vol. 26). World Scientific.


\bibitem{GGKM} Gardner, C.S., Greene, J.M., Kruskal, M.D., Miura, R.M. (1967). Method for solving the Korteweg-de Vries equation. Physical Review Letters, 19(19), 1095-1097.

\bibitem{GD} Gel'fand, I.M., Dikii, L.A. (1975). Asymptotic behaviour of the resolvent of Sturm-Liouville equations and the algebra of the Korteweg-de Vries equations. Russian Math. Surveys 30:5, 77-113.

\bibitem{GH} Gesztesy, F., Holden, H. (2003). Soliton Equations and their Algebro-Geometric Solutions: Volume 1, (1+1)-Dimensional Continuous Models. Cambridge University Press.

\bibitem{Good} Goodearl, K.R., 1983. Centralizers in differential, pseudo-differential and fractional differential operator rings. Rocky Mountain Journal of Mathematics, 13 (4), 573-618.

\bibitem{Gri} {Grigorenko, N.V. (2009). Algebraic-geometric operators and Galois differential theory. Ukrainian Mathematical Journal,  61, 14-29.}

\bibitem{Hal}{Halphen, G. H., (1886)}, Trait\'e des fonctions elliptiques et de leurs applications. Gauthier-Villars.

\bibitem{HMO} Hardouin, C., Minchenko, A. and  Ovchinnikov, O., 2017. Calculating differential Galois groups of parametrized differential equations, with applications to hypertranscendence. Mathematische Annalen, vol. 368, 587--632.

\bibitem{Harr} Harris, J. (1992). Algebraic Geometry. A First Course. Springer-Verlag New York. In series Graduate Texts in Mathematics. Volume 133.

\bibitem{Her} Hermite, C. (1912). Sur l'\'equation de Lam\'e. Oeuvres of Charles Hermite, Tome III. Gauthier-Villars, Paris.



\bibitem{Hi} Hille, J., (1976). Ordinary Differential Equations in the Complex Domain. John Wiley \& Sons.

\bibitem{Ka} Kaplansky, I., (1976). An introduction to differential algebra. Hermann; Enlarged 2nd edition.


\bibitem{Kolchin} Kolchin, E.R., (1973). Differential algebra \& algebraic groups (Vol. 54). Academic press.

\bibitem{K77} Krichever, I.M., 1977. Integration of nonlinear equations by the methods of algebraic geometry. Func. Anl. Applic. 11, 12-26.

\bibitem{K} {Krichever, I.M., 1978. Commutative rings of ordinary linear differential operators. Funct. Anal. Appl. 12, no. 3, 175-185.}

\bibitem{KN1}{I.M. Krichever, I. M., and  Novikov, S. P., 1979. Holomorphic fiberings and 
nonlinear equations. Finite zone solutions of rank 2. Dokl. Akad. Nauk SSSR (Russian) / Soviet Math. Dokl.(English), 247 /  20, 33--37 / 650--654. }

\bibitem{KN2}{I.M. Krichever, I. M., and  Novikov, S. P., 1980. Holomorphic bundles over algebraic curves and nonlinear equations. Russian Mathematical Surveys, 35:6, 53--79. }





\bibitem{Li} Li, Z., 1998. A subresultant theory for Ore polynomials with applications. Proc. Int. Symp. Symbolic and Algebraic Computation 1998 (O. Gloor, Ed.), ACM Press, 132-139.

 \bibitem{Loj} \L ojasiewicz, S., (1991). Introduction to Complex Analytic Geometry. Birkh\"auser Verlag.

\bibitem{Ma} {Matveev, V.B.2008. 30 years of finite-gap integration theory. Phil. Trans. R. Soc. A  366, 837-875.}

\bibitem{McW}{  McCallum, S., Winkler, F. (2018). Resultants: Algebraic and Differential. Techn. Rep. RISC18-08, J.Kepler University, Linz, Austria}






\bibitem{MS} Mitschi, C., Singer, M. F., 2011. Monodromy groups of parameterized linear differential equations with regular singularities. Bull. London Math. Soc. 44, 5, 913-930 .













\bibitem{Morales} Morales-Ruiz, J.J., (1999). Differential Galois theory and non-integrability of Hamiltonian systems. Birkh\"auser, Berlin.

\bibitem{Morint} {Morales-Ruiz, J.J. Picard-Vessiot theory and integrability 2015. J. Geom. and Phys. 87, 314-343.}

\bibitem{MRZ} {Morales-Ruiz, J. J., Rueda, S. L.,  Zurro, M. A. (2019). Factorization of KdV Schr\" odinger operators using differential subresultants. To appear in Advances in Applied Mathematics, 2020.}

 \bibitem{Yagasaki} {Motonaga, S., Yagasaki, K., 2018. Nonintegrability of Parametrically Forced Nonlinear Oscillators. Regul. Chaot. Dyn. 23, 291-303.}



\bibitem{Mum}{Mumford, D., 1977. An algebro-geometricc construction of commuting operators and of solutions to the
Toda lattice equation, Korteweg-de Vries equation and related non-linear ÃƒÂ©quations. Proc. Intern.
Symp. Algebraic Geometry Kyoto , 115-153.}

\bibitem{Novikov} Novikov, S. P., 1974. The periodic problem for the Korteweg-de vries equation. Functional analysis and its applications, 8(3), 236-246.












\bibitem{Olver}   Olver, P.J. (1986). Applications of Lie groups to differential equations (Vol. 107). Springer-Berlang.





\bibitem{Prev} Previato, E., 1991. Another algebraic proof of Weil's reciprocity. Atti Accad. Naz. Lincei Cl. Sci. Fis. Mat. Natur. Rend. Lincei (9) Mat. Appl. 2, no. 2, 167-171.

\bibitem{PRZ} {Previato, E., Rueda, S. L., Zurro, M. A. (2019). Commuting Ordinary Differential Operators and the Dixmier Test. Symmetry, Integrability and Geometry: Methods and Applications (SIGMA) 15 (2019), 101, 23 pages.}

\bibitem{Ritt} Ritt, J.F. (1950). Differential algebra (Vol. 33). American Mathematical Soc..




\bibitem{Sch} Schur, I., 1904. \"Uber vertauschbare lineare Differentialausdr\"ucke. Berlin Math. Gesellschaft, Sitzungsbericht 3 ( Arch. der Math., Beilage (3)8), 7-10.


\bibitem{SWP} Sendra J.R., Winkler J.R., P\'erez-D\'{\i}az S. (2007). Rational Algebraic Curves: A Computer Algebra Approach. Springer-Verlag Heidelberg. In series Algorithms and Computation  in Mathematics.  Volume 22.


\bibitem{Shafarevich} Shafarevich, I.R. (1994). Basic Algebraic Geometry 1, 2. Springer-Verlag.



\bibitem{UW} Ulmer, F., Weil, J. A., 1996. Note on Kovacic's algorithm. Journal of Symbolic Computation, 22(2), 179-200.

\bibitem{VPS} Van der Put, M., Singer, M.F. (2012). Galois theory of linear differential equations (Vol. 328). Springer Science \& Business Media.

\bibitem{Veselov} Veselov, A. P., 2011. On Darboux-Treibich-Verdier Potentials. Letters in Mathematical Physics, 96(1), 209-216.

\bibitem{We}  Weikard, R., 2000. On commuting differential operators. Electron. J. Differential Equations Vol. 2000, N. 19, 1-11.

 
  \bibitem{WW}  Whittaker, E. T. and Watson, G. N., 1996. A course of modern analysis. Cambridge university press.


 \bibitem{Wilson}  Wilson, G., 1985. Algebraic curves and soliton equations. In {\it Geometry Today}, eds. E. Arbarello et al., Birkh\"auser, Boston, 303-329.






\end{thebibliography}

\end{document}